\documentclass[a4paper,12pt]{amsart}

\usepackage{float}
\usepackage{euscript,verbatim}
\usepackage{eufrak}
\usepackage{graphicx}
\usepackage[usenames]{color}
\usepackage[colorlinks,linkcolor=red,anchorcolor=blue,citecolor=blue]{hyperref}
\usepackage{amsmath}
\usepackage{amsthm}
\usepackage[all]{xy}
\usepackage{graphicx}
\usepackage{amssymb}
\usepackage{dsfont}
\usepackage{enumerate}

\newtheorem*{thm*}{Theorem}
\newtheorem{thm}{Theorem}[section]
\newtheorem{prop}[thm]{Proposition}

\newtheorem{lem}[thm]{Lemma}
\newtheorem{cor}[thm]{Corollary}

\newtheorem{ex}[thm]{Example}

\def\N{\mathbb{N}}

\def\N{\mathbb{N}}
\def\R{\mathbb{R}}
\def\DD{\mathcal{D}}
\def\AA{\mathcal{A}}
\def\ind{\mathds{1}}

\newcommand{\ii}{\mathbf{i}}
\newcommand{\ww}{\mathbf{w}}
\newcommand{\jj}{\mathbf{j}}
\newcommand{\zz}{\mathbf{z}}
\newcommand{\p}{{\underline{p}}}

\makeatletter 
\@addtoreset{equation}{section}
\numberwithin{equation}{section}

\begin{document}
\title[Weighted Birkhoff averages]{On the multifractal spectrum of weighted Birkhoff averages}

\author{Bal\'azs B\'ar\'any}
\address[Bal\'azs B\'ar\'any]{Budapest University of Technology and Economics, Department of Stochastics, MTA-BME Stochastics Research Group, P.O.Box 91, 1521 Budapest, Hungary}
\email{balubsheep@gmail.com}

\author{Micha\l\ Rams}
\address[Micha\l\ Rams]{Institute of Mathematics, Polish Academy of Sciences, ul. \'Sniadeckich 8, 00-656 Warszawa, Poland}
\email{rams@impan.pl}

\author{Ruxi Shi}
\address[Ruxi Shi]{Institute of Mathematics, Polish Academy of Sciences, ul. \'Sniadeckich 8, 00-656 Warszawa, Poland}
\email{rshi@impan.pl}

\thanks{Bal\'azs B\'ar\'any acknowledges support from grants OTKA K123782 and OTKA~FK134251. Micha\l\ Rams was supported by National Science Centre grant 2019/33/B/ST1/00275 (Poland).}

\maketitle
\begin{abstract}
	In this paper, we study the topological spectrum of weighted Birkhoff averages over aperiodic and irreducible subshifts of finite type. We show that for a uniformly continuous family of potentials, the spectrum is continuous and concave over its domain. In case of typical weights with respect to some ergodic quasi-Bernoulli measure, we determine the spectrum.  Moreover, in case of full shift and under the assumption that the potentials depend only on the first coordinate, we show that our result is applicable for regular weights, like M\"obius sequence.
\end{abstract}
\date{\today}

\thispagestyle{empty}

	\maketitle
	
\section{Introduction}

Let $T\colon X\mapsto X$ be a measure preserving transformation of the standard  Borel  probability space $(X,\mathcal{B},\nu)$. The well-known Theorem of Birkhoff states that for any $f\in L^1(X,\mathcal{B},\nu)$, the limit
$$
\lim_{n\to\infty}\frac{1}{n}\sum_{k=0}^{n-1}f(T^kx)\text{ exists for $\nu$-almost every $x$.}
$$
Moreover, if $\nu$ is an ergodic measure with respect to $T$ then
$$
\lim_{n\to\infty}\frac{1}{n}\sum_{k=0}^{n-1}f(T^kx)=\int fd\nu\text{ for $\nu$-almost every $x$.}
$$
The ``time'' averages $\frac{1}{n}\sum_{k=0}^{n-1}f(T^kx)$ are called the Birkhoff averages. If $T$ is uniquely ergodic then for a continuous potential $f\in C(X)$, the limit $\lim_{n\to\infty}\frac{1}{n}\sum_{k=0}^{n-1}f(T^kx)$ exists for all $x\in X$ and converges to a constant. That is, there is no multifractal behaviour. However, if the system $(X,T)$ has a large family of ergodic measures (for example, a full shift), one may expect that the limit of the Birkhoff averages can take a wide variety of values. It is a natural question to ask, how large (for example, topological entropy, Hausdorff or packing dimension) is the set of points in $X$ for which the Birkhoff average converges to a prescribed value $\alpha$? It leads to the multifractal analysis and there has been a considerable amount of works on this.

As far as we know, the first work is due to Besicovitch \cite{B1935} where he studied the Hausdorff dimension of sets given by the frequency of digits in dyadic expansions. Then it was
subsequently extended by Eggleston \cite{E1949}. For further results on digit frequencies, see Barreira, Saussol and Schmeling \cite{BSS}. For multifractal analysis of Birkhoff averages, we refer to \cite{FF1998, FFW2001, O1998, S1999, FLW2002, FLP2008, FF2000, BSS2002} and references therein.

Let $X$ be a compact metric space, let $T\colon X\mapsto X$ be a continuous transformation, and let $\varphi\colon X\mapsto\R$ be a continuous potential. Takens and Verbitskiy \cite{TV} showed that for an $\alpha\in\R$, the topological entropy of the set
$$
E(\alpha)=\left\{x\in X:\lim_{n\to\infty}\frac{1}{n}\sum_{k=0}^{n-1}\varphi(T^kx)=\alpha\right\}
$$
equals to the Legendre transform of the topological pressure, which is equal to the supremum of the entropy of all invariant and ergodic measures for which the ``space'' average (i.e. the integral of $\varphi$) equals to $\alpha$. 

In this paper, we are interested in the generalisation of the problem above for weighted Birkhoff averages. Let $\ww=\{w_k\}_{k\in\N}$ be a sequence of bounded reals and let $\varphi\colon X\mapsto\R^d$ be a continuous potential and let $\alpha\in\R^d$. Is it possible to determine
$$
h_{\rm top}\left(\left\{x\in X:\lim_{n\to\infty}\frac{1}{n}\sum_{k=0}^{n-1}w_k\varphi(T^kx)=\alpha\right\}\right)=?
$$

Weighted Birkhoff averages were studied since 1940', since the celebrated Theorem of Wiener and Wintner. Since then several generalisations of the weighted ergodic theorems appeared, see for example \cite[Chapter~21]{EFHNbook}. The first study of weighted Birkhoff averages in the context of spectrum was by Fan \cite{F3}. Lately, there were an additional motivation for this problem by Sarnak's conjecture \cite{S}. Let us recall the definition of the M\"obius sequence, $\boldsymbol{\mu}\colon\N\mapsto\{-1,0,1\}$,
$$
\boldsymbol{\mu}(n)=\begin{cases} (-1)^k & \text{ if $n$ is a product of $k$ distinct primes,}\\
0 & \text{ if there exists $a\geq2$ such that $a^2\vert n$.}
\end{cases}
$$
Sarnak's conjecture \cite{S} claims that if $T\colon X\mapsto X$ is continuous over the compact metric space $X$ with topological zero entropy  then for every $x\in X$ and every continuous potential $\varphi\colon X\mapsto\R$
$$
\lim_{n\to\infty}\frac{1}{n}\sum_{k=0}^{n-1}\boldsymbol{\mu}(k)\varphi(T^kx)=0.
$$
Even though Sarnak's conjecture has been verified for various special dynamical systems (e.g. rotations on the circle, automorphism of the torus with entropy zero etc.), it is still widely open in general. We refer to \cite{FKL2018} for a survey of many recent results on Sarnak conjecture.
El Abdalaoui, Ku\l aga-Przymus, Lema\'nczyk and de la Rue \cite{AKLd} showed Birkhoff's type ergodic theorem with M\"obius weight.
\begin{thm*}[\cite{AKLd}]
	Let $T$ be  an  automorphism  of  a  standard  Borel  probability space $(X,\mathcal{B},\nu)$ and let $f\in L^1(X,\mathcal{B},\nu)$.  Then, for $\nu$-almost every $x\in X$, we have
	$$
	\lim_{n\to\infty}\frac{1}{n}\sum_{k=0}^{n-1}\boldsymbol{\mu}(k)f(T^{k}(x))=0.
	$$
\end{thm*}
Fan \cite{F} proved a similar result for a more general family of sequences, like Davenport's type. Hence, the usual method of calculating the spectrum  for weighted Birkhoff averages, that is, to show that it is equal to the supremum of the entropy of invariant measures, is not applicable. This paper is devoted to present a method, which allows us to calculate the spectrum. Recently, Fan \cite{F2} studied the same question, but with strictly different methods. We will point out the main differences between our and his results.

\section{Results}

In the rest of the paper, we restrict our interest to the full shift space. That is, let $\mathcal{A}=\{1,\ldots,K\}$ be a finite alphabet, and let $\Sigma=\mathcal{A}^\N$. Let us denote the left-shift operator on $\Sigma$ by $\sigma$. Denote $\Sigma_n$ the set of $n$-length finite word. Moreover, denote $\Sigma_*$ the set of all finite prefixes of the infinite words in $\Sigma$. For an $\ii=(i_0,i_1,\ldots)\in\Sigma$ and $m>n\geq0$ let $\ii|_n^m=(i_n,\ldots,i_m)$ be the subword of $\ii$ between the positions $n$ and $m$, and for short denote by $\ii|_n$ the first $n$ element of $\ii$, i.e. $\ii|_n=\ii|_{0}^{n-1}$. For an $\ii\in\Sigma_*$, denote $|\ii|$ the length of $\ii$ and let $[\ii]$ denote the corresponding cylinder set, that is, $[\ii]:=\{\jj\in\Sigma:\jj|_{|\ii|}=\ii\}$. We use $l(\cdot)$ to denote the level of cylinder. Moreover, The space $\Sigma$ is clearly metrisable with metric
\begin{equation}\label{eq:metric}
d(\ii,\jj)=e^{-\min\{n\geq0:i_n\neq j_n\}}.
\end{equation}

In some cases, we extend our interest to a special family of $\sigma$-invariant compact sets. Let $\mathbf{A}$ be a $K\times K$ matrix with entries $0,1$, and we say that the set $\Sigma_{\mathbf{A}}\subseteq\Sigma$ is {\it subshift of finite type} if
$$
\Sigma_{\mathbf{A}}=\{\ii=(i_0,i_1,\ldots)\in\mathcal{A}^\N:\mathbf{A}_{i_k,i_{k+1}}=1\text{ for every }k=0,1,\ldots\}.
$$
We call the matrix $\mathbf{A}$ the adjacency matrix. Let us denote the set of admissible words with length $n$ (i.e. $n$-length subwords of some element in $\Sigma_{\mathbf{A}}$) by $\Sigma_{\mathbf{A},n}$ and denote $\Sigma_{\mathbf{A},*}$ the set of all admissible words. Without loss of generality, we may assume that $\Sigma_{\mathbf{A},1}=\mathcal{A}$. Moreover, we say that $\Sigma_{\mathbf{A}}$ is {\it aperiodic and irreducible} if there exists $r\geq1$ such that every entry of $\mathbf{A}^r$ is strictly positive.

\subsection{Topological entropy}\label{sec:Topological entropy} 

Before we turn to our main results, let us recall here the definition of topological entropy on the shift space. Let $\Sigma=\mathcal{A}^\N$ be the symbolic space. Let $E\subset \Sigma$.
	Define
	$$\mathcal{H}^s_r(E):=\inf_{\alpha} \sum_{C\in \alpha} e^{-sl(C)}$$
	where $\alpha$ is taken over all covers consisting of cylinders of levels large than $r$. Clearly, $\mathcal{H}^s_r(E)$ is increasing as a function of $r$. We define
	$$\mathcal{H}^s(E):=\lim_{r\to \infty}\mathcal{H}^s_r(E)\in [0,+\infty].$$
	
	The \textit{topological entropy} of $E$ is the value where the above limit jumps from $+\infty$ to 0, that is,
	$$
	h_{\rm top}(E):=\inf\{ s\ge 0:\mathcal{H}^s(E)<+\infty \}.
	$$
	An upper bound of $h_{\rm top}(E)$ is given by
	\begin{equation}\label{eq:topentbasic}
		h_{\rm top} (E) \leq \liminf_{n\to\infty} \frac 1n \log \#\{\ii\in\Sigma_n:[\ii]\cap E\neq\emptyset\}.
	\end{equation}
	In fact, the reason is that we can always take a cover with cylinders of level $n$ when estimating $\mathcal{H}_n^s(E)$. If $E$ is a closed $\sigma$-invariant set, then the equality holds (see for example \cite[Theorem 2.6]{KR}). However, the equality does not necessarily hold in general, because there might exist a better cover (in the sense that we could get smaller value of $\sum_{C\in \alpha} e^{-sl(C)}$) than covers consisting of cylinders of level $n$.

	
	To get the lower bound, one has a version of Frostman Lemma as follows.
	\begin{lem}\label{lem:Forstman}
		Let $E\subset \Sigma$. Suppose that there exists a probabilistic measure $\mu$ on $E$ satisfying that there is a constant $c$ such that for every cylinder $C$, we have $\mu(C\cap E) \leq ce^{-sl(C)}$. Then $h_{\rm top} (E) \geq s$.
\end{lem}

\subsection{Continuity of the entropy}

The first aspect of the study is the continuity of the entropy in a more general setting than weighted Birkhoff averages. That is, let $\Sigma_{\mathbf{A}}$ be an aperiodic and irreducible subshift of finite type and let $\phi_i:\Sigma_{\mathbf{A}}\to\R$ be a sequence of continuous potentials. We say that the sequence of potentials $\{\phi_i\}$ are {\em uniformly equicontinuous} if
\begin{equation}\label{def:rho1}
\rho_n^{(1)} := \sup_i {\rm var}_n(\phi_i),
\end{equation}
is finite for every $n$ and converges to 0 as $n$ tends to $\infty$, where
$$
{\rm var}_n(\phi) := \sup_{\ii\in \Sigma_n} \sup_{\jj,\mathbf{k}\in[\ii]} \{ |\phi(\jj)-\phi(\mathbf{k})|\}.
$$
For $\ii\in \Sigma_{\mathbf{A}}$, let
$$
\overline{A}(\ii):= \limsup_{n\to\infty} \frac 1n \sum_{i=0}^{n-1} \phi_i(\sigma^i \ii),
$$
$$
\underline{A}(\ii):= \liminf_{n\to\infty} \frac 1n \sum_{i=0}^{n-1} \phi_i(\sigma^i \ii).
$$
Moreover, if the limit exists let
$$
A(\ii):= \lim_{n\to\infty} \frac 1n \sum_{i=0}^{n-1} \phi_i(\sigma^i \ii).
$$

Given $\alpha\leq\beta\in\R$, let
\[
L_{\mathbf{A}}(\alpha,\beta)=\{\ii\in\Sigma_{\mathbf{A}}: \underline{A}(\ii)=\alpha\text{ and }\overline{A}(\ii)=\beta \}.
\]
For short, let $L_{\mathbf{A}}(\alpha):=L_{\mathbf{A}}(\alpha,\alpha)$. Now we state our first main result. 

\begin{thm}\label{thm:cont}
Let $\Sigma_{\mathbf{A}}\subseteq \Sigma$ be an aperiodic and irreducible subshift of finite type. For every sequence $\phi_i\colon\Sigma_{\mathbf{A}}\mapsto\R$ of uniformly equicontinuous potentials, the function $\alpha\mapsto h_{\rm top}(L_{\mathbf{A}}(\alpha))$ is continuous and concave over its domain, which is a (possibly empty) closed interval.
\end{thm}

In his paper, Fan \cite{F2} gave upper and lower bounds for $h_{\rm top}(L_{\mathbf{A}}(\alpha))$ in case of full shift by using a generalized topological pressure generated by the sequence $\phi_i$. If the pressure is sufficiently smooth then these bounds agree.

It is a natural question how large is the set of irregular points, that is, let
$$
D:=\left\{\ii\in\Sigma_{\mathbf{A}}:\lim_{n\to\infty}\frac{1}{n}\sum_{k=0}^{n-1}\phi_i(\sigma^k\ii)\text{ does not exists }\right\}.
$$

\begin{thm} \label{thm:contgen}
Let $\Sigma_{\mathbf{A}}\subseteq \Sigma$ be an aperiodic and irreducible subshift of finite type. Let $\phi_i\colon\Sigma_{\mathbf{A}}\mapsto\R$ be a sequence of uniformly equicontinuous potentials.
Assume that $A(\ii)$ takes at least two possible values, that is, the domain of the function $\alpha\mapsto h_{\rm top}(L_{\mathbf{A}}(\alpha))$ is a nontrivial interval. Then
\[
h_{\rm top}(D) = h_{\rm top}(\Sigma_{\mathbf{A}}).
\]
\end{thm}

\subsection{Random weights} Let us now extend our symbolic space $\Sigma=\mathcal{A}^\N$. Namely, Let $\Lambda=\{1,\ldots,N\}$ be another finite alphabet, and let $\Omega=\Lambda^\N$ be compact left-shift invariant subsets. Let us define the extended symbolic space $\Gamma:=\Omega\times\Sigma$. As an abuse of notation, we denote the left-shift operator on $\Omega$, and $\Gamma$ by $\sigma$ too. Adapting the notations for $\Omega$ and $\Gamma$, let $\Omega_n$ and $\Gamma_n$ be the set of $n$-length finite words, and denote $\Omega_*$ and $\Gamma_*$ the set of all finite words. The spaces $\Omega,\Sigma$ and $\Gamma$ are clearly metrisable with the same metric defined in \eqref{eq:metric}. For short, denote $\ii\wedge\jj=\min\{n\geq0:i_n\neq j_n\}$.

For an aperiodic and irreducible subshift of finite type $\Sigma_{\mathbf{A}}\subseteq\Sigma$, the set $\Gamma_{\mathbf{A}}=\Omega\times\Sigma_{\mathbf{A}}$ is an aperiodic and irreducible subshift of finite type as well. Denote the set of finite admissible words by $\Gamma_{\mathbf{A},*}$ and that of words of length $n$ by $\Gamma_{\mathbf{A},n}$. Let $f\colon\Gamma_{\mathbf{A}}\mapsto\R^d$ be a continuous potential. For a given sequence $\ww\in\Omega$ and $\alpha\in\R^d$ let
$$
E_\ww(\alpha):=\left\{\ii\in\Sigma_{\mathbf{A}}:\lim_{n\to\infty}\frac{1}{n}\sum_{k=0}^{n-1}f(\sigma^k\ww,\sigma^k\ii)=\alpha\right\}.
$$
Our goal is to determine the topological entropy of $E_\ww(\alpha)$, at least for the case of typical $\ww\in\Omega$. In order to do so, we need to introduce further regularity properties on $f$ and on the choice of $\ww$.

We say that the potential $f\colon\Omega\times\Sigma_{\mathbf{A}}\mapsto\R^d$ has {\it summable variation} if
$$
\sum_{k=0}^\infty\max_{\substack{(\ww,\ii),(\zz,\jj)\in\Gamma_{\mathbf{A},*}:\\(\ww,\ii)\wedge(\zz,\jj)=k}}\|f(\ww,\ii)-f(\zz,\jj)\|<\infty.
$$

Let $\nu$ be a $\sigma$-invariant ergodic measure on $\Omega$. We say that $\nu$ is {\em quasi-Bernoulli} if there exists $C>0$ such that for every $\ww,\zz\in\Omega_*$ with $\ww\zz\in\Omega_*$
$$
C^{-1}\nu([\ww])\nu([\zz])\leq\nu([\ww\zz])\leq C\nu([\ww])\nu([\zz]).
$$
Denote by $\Pi$ the natural projection $\Pi\colon\Omega\times\Sigma\mapsto\Omega$, that is, $\Pi(\ww,\ii)=\ww$. Denote by $\mathcal{E}_\nu(\Gamma)$, $\mathcal{E}_\nu(\Gamma_{\mathbf{A}})$ the set of ergodic $\sigma$-invariant measures on $\Gamma$ and $\Gamma_{\mathbf{A}}$ respectively, whose marginal is $\nu$, i.e., $\Pi_*\mu=\nu$. Denote by $\mathcal{M}_\nu(\Gamma)$ and $\mathcal{M}_\nu(\Gamma_{\mathbf{A}})$ the set of $\sigma$-invariant measures on $\Gamma$ and $\Gamma_{\mathbf{A}}$ with marginal $\nu$. Let
\begin{equation}\label{eq:defpa}
\mathcal{P}_{\mathbf{A}}=\{\alpha\in\R^d:\text{ there exists }\mu\in\mathcal{M}_\nu(\Gamma_{\mathbf{A}})\text{ such that }\int f d\mu=\alpha\}.
\end{equation}
Denote the relative interior of $\mathcal{P}_{\mathbf{A}}$ by $\mathcal{P}_{\mathbf{A}}^o$.

Moreover, let us define the conditional pressure of a potential $f:\Gamma_{\mathbf{A}}\mapsto\R$ by
\begin{equation}\label{eq:condpresdefdef}
P_{\nu}(f)=\lim_{n\to\infty}\frac{1}{n}\int\log\sum_{\ii\in\Sigma_n}\sup_{\jj\in[\ii]}e^{S_nf(\ww,\jj)}d\nu(\ww),
\end{equation}
where $S_nf=f+f\circ\sigma+\cdots+f\circ\sigma^{n-1}$ and $\log$ is taken in the base $e$. Throughout the paper, we will use the convention that $0\cdot\log0=0$. Moreover, we note that we define the supremum over an empty set as $-\infty$ and the topological entropy of an empty set as $-\infty$.
Now, we can formalise our second theorem.

\begin{thm}\label{thm:typmain} Let $\Sigma_{\mathbf{A}}\subseteq\Sigma$ be an aperiodic and irreducible subshift of finite type, and let $\nu$ be a quasi-Bernoulli $\sigma$-invariant ergodic measure on $\Omega$. Moreover, let $f\colon\Omega\times\Sigma_{\mathbf{A}}\mapsto\R^d$ be a continuous map with summable variation. Then for every $\alpha\in\mathcal{P}^o_{\mathbf{A}}$ and for $\nu$-almost every $\ww\in\Omega$,
	\[
	\begin{split}
	h_{\rm top}(E_{\ww}(\alpha))&=\sup\{h_\mu:\mu\in\mathcal{E}_\nu(\Gamma_{\mathbf{A}})\text{ and }\int fd\mu=\alpha\}-h_\nu\\
	&=\sup\{h_\mu:\mu\in\mathcal{M}_\nu(\Gamma_{\mathbf{A}})\text{ and }\int f d\mu=\alpha\}-h_\nu\\
	&=\inf_{\p\in\R^d}P_\nu(\langle\p,f-\alpha\rangle).
	\end{split}
	\]	
	Furthermore, there exists $\alpha_0\in\R^d$ such that for $\nu$-almost every $\ww$,
	\begin{equation}\label{eq:max}
	h_{\rm top}(E_\ww(\alpha_0))=h_{\rm top}(\Sigma_{\mathbf{A}}).
	\end{equation}
\end{thm}

Combining Theorem~\ref{thm:cont} and Theorem~\ref{thm:typmain} we get the following stronger result for real valued potentials, which shows that for a typical sequence of weights it is possible to calculate the whole spectrum.

\begin{thm}\label{cor:main}
Let $\Sigma_{\mathbf{A}}\subseteq\Sigma$ be an aperiodic and irreducible subshift of finite type, and let $\nu$ be a quasi-Bernoulli $\sigma$-invariant ergodic measure on $\Omega$. Moreover, let $f\colon\Omega\times\Sigma_{\mathbf{A}}\mapsto\R$ be a continuous map with summable variation. Then for $\nu$-almost every $\ww\in\Omega$,
\[
\begin{split}
h_{\rm top}(E_{\ww}(\alpha))&=\sup\{h_\mu:\mu\in\mathcal{E}_\nu(\Gamma_{\mathbf{A}})\text{ and }\int f d\mu=\alpha\}-h_\nu\\
&=\sup\{h_\mu:\mu\in\mathcal{M}_\nu(\Gamma_{\mathbf{A}})\text{ and }\int f d\mu=\alpha\}-h_\nu\\
&=\inf_{p\in\R}\left(P_\nu(p\cdot f)-\alpha\cdot p\right)\text{ for every $\alpha\in\R$}.
\end{split}
\]
Moreover, for $\nu$-almost every $\ww$, the map $\alpha\mapsto h_{\rm top}(E_\ww(\alpha))$ is continuous and concave over its domain $\mathcal{P}_{\mathbf{A}}$.
\end{thm}

Fan~\cite{F2} proved some similar results. Namely, he showed a version of Theorem~\ref{cor:main} for full shifts with the choice $f(\sigma^k\ww,\sigma^k\ii)=w_k\varphi(\ii)$, where $(w_k)_k$ is an ergodic sequence of real random variables or deduced from a uniquely ergodic dynamical system, and $\varphi$ depends only on a finite number of coordinates. In this cases, he shows analyticity of the conditional topological pressure, while our result only gives continuity.



\subsection{Potentials depending on the first coordinate} Let us assume that $f\colon\Omega\times\Sigma\mapsto\R$ depends only on the first symbol, that is, $f(\ww,\ii)=f_{w_0,i_0}$. Then for a $\ww\in\Omega$,
$$
E_\ww(\alpha):=\left\{\ii\in\Sigma:\lim_{n\to\infty}\frac{1}{n}\sum_{k=0}^{n-1}f_{w_k,i_k}=\alpha\right\}.
$$

Let $\underline{q}=(q_1,\ldots,q_N)\in\mathcal{S}_N$ be a probability vector, where $\mathcal{S}_N$ denotes the $(N-1)$-dimensional simplex. We say that $\ww\in\Omega$ is {\it $\underline{q}$-frequency regular}, if
\begin{equation}\label{eq:freq}
\lim_{n\to\infty}\frac{\#\{k\in[0,n]\cap\mathbb{Z}:\omega_k=i\}}{n}=q_i\text{ for every }i=1,\ldots,N.
\end{equation}

Notice that there is a bijection between $\mathcal{S}_N$ and the probability Bernoulli measure on $\Omega$.
In this case, we choose $\nu=\nu_{\underline{q}}$ to be the Bernoulli measure on $\Omega$. Then for the potential $f\colon\Gamma\mapsto\R$, the conditional pressure has the form
\begin{equation}\label{eq:simplepres}
P_{\underline{q}}(\langle\p,f-\alpha\rangle)=\sum_{j=1}^Nq_j\log\sum_{i=1}^Ke^{\langle\p,f_{j,i}-\alpha\rangle}.
\end{equation}

Denote by $\mathcal{B}_{\underline{q}}(\Gamma)$ the set of all Bernoulli measures on $\Gamma$ with marginal $\nu$. That is, let $(p_{j,i})_{j=1,i=1}^{N,K}\in\mathcal{B}_{\underline{q}}(\Gamma)\subset\mathcal{S}_{NK}$ such that $\sum_{i=1}^Kp_{j,i}=q_j$. Our third main result is as follows.

\begin{thm}\label{thm:goal} Let $\ww\in\{1,\ldots,N\}^\N$ be a $\underline{q}$-frequency regular sequence with frequencies $(q_1,\ldots,q_N)$. Then for every $\alpha\in\R$.
\[
\begin{split}
h_{\rm top}(E_\ww(\alpha))&=\sup_{(p_{j,i})\in\mathcal{B}_{\underline{q}}(\Gamma)}\left\{-\sum_{i,j}p_{j,i}\log p_{ji}:\sum_{i,j}p_{j,i}f_{j,i}=\alpha\right\}+\sum_{i=1}^Nq_i\log q_i\\
&=\inf_{p\in\R}\left\{P_{\underline{q}}(p\cdot f)-p\alpha\right\}.
\end{split}
\]
\end{thm}

Comparing Theorem~\ref{thm:typmain} with Theorem~\ref{cor:main}, in the general setup of Theorem~\ref{thm:typmain} we are only able to show that for any possible value of $\alpha$ one can find a full measure set $\Omega_\alpha$, which might depend on $\alpha$, while in Theorem~\ref{cor:main} in the one dimensional case, we manage to show that there exists a universal full measure set $\Omega$, for which any $\ww\in\Omega$ the spectrum $\alpha\mapsto h_{\rm top}(E_\ww(\alpha))$ can be determined. Comparing Theorem \ref{thm:goal} with Theorem \ref{cor:main}, we can construct a Bernoulli measure with probabilities $\underline{q}=(q_i)_i$, for which the $\underline{q}$-frequency regular sequences will be a set of full measure. That is, we can explicitly construct the set $\Omega$ of full measure for which the spectrum can be determined. In this sense Theorem \ref{thm:goal} is a strengthening of Theorem \ref{cor:main} for this particular class of systems.


Fan \cite{F2} also gave a similar result in his recent preprint. Namely, Fan shows Theorem~\ref{thm:goal} under a weaker condition that $\varphi$ depends on finitely many coordinates but under the stronger assumption that it takes only values ${-1,1}$.

Now we state the corresponding version of Theorem~\ref{thm:contgen} for the frequency regular case. Similarly, let
$$
D_\ww=\left\{\ii\in\Sigma:\lim_{n\to\infty}\frac{1}{n}\sum_{k=0}^{n-1}f_{w_k,i_k}\text{ does not exists}\right\}.
$$

\begin{thm}\label{thm:irreg2} Let $\ww\in\{1,\ldots,N\}^\N$ be a $\underline{q}$-frequency regular sequence with frequencies $(q_1,\ldots,q_N)$. Suppose that $g_i=\sum_{j=1}^Nq_jf_{j,i}$ is not constant as function of $i$. Then
	$$
	h_{\rm top}(D_\ww)=\log K.
	$$
\end{thm}

\subsection{Examples: weighted Birkhoff averages with frequency regular weights} Now, we show examples and demonstrate our result on the spectrum of real valued potentials depending on the first coordinate and frequency regular weights. Here we assume again that our potentials are supported on the whole spaces $\Sigma=\mathcal{A}^\N$, $\Omega=\Lambda^\N$ and the potential $\varphi\colon\Sigma\mapsto\R$ and the weight $\lambda\colon\Omega\mapsto\R$ depend only on the first symbol, that is, $\varphi(\ii)=\varphi_{i_0}$ and $\lambda(\ww)=\lambda_{w_0}$. Then for a $\ww\in\Omega$, let
$$
E_\ww(\alpha):=\left\{\ii\in\Sigma:\lim_{n\to\infty}\frac{1}{n}\sum_{k=0}^{n-1}\lambda_{w_k}\varphi_{i_k}=\alpha\right\}.
$$

Let $\underline{q}=(q_1,\ldots,q_N)\in\mathcal{S}_N$ be a probability vector, and let $\ww\in\Omega$ be an arbitrary $\underline{q}$-frequency regular sequence. Denote by $\varphi_{\max}=\max\{\varphi_i: 1\le i\le K \}$ and $\varphi_{\min}=\min\{\varphi_i: 1\le i\le K \}$. To avoid the trivial case, we assume $\varphi_{\max}\not=\varphi_{\min} $. Finally, let
\begin{equation}\label{eq:domain}
I=\left[\varphi_{\min}\sum_{\lambda_{j}>0}q_j\lambda_{j}+\varphi_{\max}\sum_{\lambda_{j}<0}q_j\lambda_{j}, \varphi_{\max}\sum_{\lambda_{j}>0}q_j\lambda_{j}+\varphi_{\min}\sum_{\lambda_{j}<0}q_j\lambda_{j}\right].
\end{equation}
Now we show a compatible form of $h_{\rm top}(E_\ww(\alpha))$ in order to compute some examples.

\begin{ex}\label{thm: n=1}
	Let $\ww\in\{1,\ldots,N\}^\N$ be a $\underline{q}$-frequency regular sequence with frequencies $(q_1,\ldots,q_N)$. Then for every $\alpha\in I$
	\[
	h_{\rm top}(E_\ww(\alpha))=
	\sum_{j=1}^Nq_j\log\sum_{i=1}^Ke^{p(\lambda_{j}\varphi_i-\alpha)},
	\]
	where $p$ is the unique solution of the equation
	\begin{equation}
	\sum_{j=1}^{N} q_j\lambda_j \frac{\sum_{i=1}^{N} \varphi_ie^{p\lambda_j \varphi_i} }{\sum_{i=1}^{N}e^{p\lambda_j \varphi_i}}=\alpha.
	\end{equation}
	Moreover, if $\alpha\notin I,$ $\inf_{p}P_\nu(f_{p})=-\infty$, that is, there is no $\p^*\in\R^d$ such that $\inf_{p}P_\nu(f_{p})=P_\nu(f_{p^*})$.
\end{ex}

\begin{proof}
	Let $\alpha\in I$.	For sake of simplicity, denote
	$P(p)=P_{\underline{q}}(p(\lambda\varphi-\alpha))$. It is easy to check by \eqref{eq:simplepres} that
	$$
	P(p)=\sum_{j=1}^{N}q_j\log \sum_{i=1}^{N}e^{p\lambda_j \varphi_i}-p\alpha.
	$$
	It follows that
	$$
	P'(p)=\sum_{j=1}^{N}q_j \lambda_j \frac{\sum_{i=1}^{N} \varphi_ie^{p\lambda_j \varphi_i} }{\sum_{i=1}^{N}e^{p\lambda_j \varphi_i}}-\alpha,
	$$
	and
	$$
	P''(p)=\sum_{j=1}^{N}q_j \lambda_j^2 \frac{(\sum_{i=1}^{N} \varphi_i^2e^{p\lambda_j \varphi_i})(\sum_{i=1}^{N}e^{p\lambda_j \varphi_i})-(\sum_{i=1}^{N} \varphi_ie^{p\lambda_j \varphi_i})^2 }{(\sum_{i=1}^{N}e^{p\lambda_j \varphi_i})^2}.
	$$
	Since $\varphi_{\max}\not=\varphi_{\min}$, by Cauchy-Schwarz inequality, we see that $P''(p)>0$ for all $p\in \R$. A simple computation shows that
	$$
	P'(-\infty)=\varphi_{\min}\sum_{\lambda_{j}>0}q_j\lambda_{j}+\varphi_{\max}\sum_{\lambda_{j}<0}q_j\lambda_{j}-\alpha<0,
	$$
	and
	$$P'(+\infty)=\varphi_{\max}\sum_{\lambda_{j}>0}q_j\lambda_{j}+\varphi_{\min}\sum_{\lambda_{j}<0}q_j\lambda_{j}-\alpha>0.
	$$
	Thus $P'(p)=0$ has a unique solution at which $P$ achieves minima.
	
	Now let $\alpha\notin I$. It is easy to calculate that
	$$
	P(-\infty)=\lim\limits_{p\to -\infty} p(\varphi_{\min}\sum_{\lambda_{j}>0}q_j\lambda_{j}+\varphi_{\max}\sum_{\lambda_{j}<0}q_j\lambda_{j}-\alpha),
	$$
	and
	$$
	P(+\infty)=\lim\limits_{p\to -\infty} p(\varphi_{\max}\sum_{\lambda_{j}>0}q_j\lambda_{j}+\varphi_{\min}\sum_{\lambda_{j}<0}q_j\lambda_{j})-\alpha).
	$$
	Thus $\inf_{p}P(p)=-\infty$.
\end{proof}

\begin{ex}
Let us consider again the M\"obius sequence with the potential $\varphi(\ii)=i_0$ for $\ii\in\Sigma=\{0,\ldots,N-1\}^\N$. The M\"obius function is frequency regular with
\[
\begin{split}
\lim_{n\to\infty}\large\frac{\#\{0\leq i\leq n-1:\boldsymbol{\mu}(i)=\pm1\}}{n}&=\frac{3}{\pi^2}\text{ and }\\
\lim_{n\to\infty}\frac{\#\{0\leq i\leq n-1:\boldsymbol{\mu}(i)=0\}}{n}&=1-\frac{6}{\pi^2},
\end{split}
\]
see for example \cite{CS}. As a special case of Example~\ref{thm: n=1} for $\varphi\colon\{0,\ldots,N-1\}^\N\mapsto\R$ with $\varphi(\ii)=i_0$, we get
$$
h_{\rm top}(E_{\boldsymbol{\mu}}(\alpha))=\left(1-\frac{6}{\pi^2}\right)\log(N)+\frac{6}{\pi^2}\log\left(\frac{e^{pN}-1}{e^p-1}\right)-\left((N-1)\frac{3}{\pi^2}+\alpha\right)p,
$$
where $H(\underline{p})=-\sum_ip_i\log p_i.$ and $p$ is the unique solution of
$$
\frac{(e^{(N+1)p}-1)(N-1)-(N+1)(e^{N p}-e^p)}{(e^{Np}-1)(e^p-1)}=\frac{\pi^2\alpha}{3},\text{ for }\alpha\in\left[\frac{-(N-1)3}{\pi^2},\frac{(N-1)3}{\pi^2}\right].
$$
\end{ex}

A corollary of the above results is that non-degenerate weights and potentials give us non-degenerate weighted spectrum.

\begin{cor}
	Let $\ww\in\Omega$ be a frequency regular sequence with frequencies $(q_1,\ldots,q_N)$ with non-degenerate weights, i.e. $\sum_{j=1}^Nq_j|\lambda_j|>0$. Let $\varphi\colon\Sigma\mapsto\R$ be a potential depending only on the first coordinate. Then there exists $\alpha_0\in I$ such that $h_{\rm top}(E_\ww(\alpha_0))=\log K$. Moreover, the domain $I$ is a non-degenerate closed interval unless the potential $\varphi(\ii)=\varphi_{i_0}$ is constant. In particular, either the limit of the weighted Birkhoff average at every point exists and equals $\alpha_0$ or the set of points at which the limit of the weighted Birkhoff average does not exist has full topological entropy.
\end{cor}

\begin{proof}
	The first assertion follows by Theorem~\ref{thm:goal} for $f_{j,i}=\lambda_j\varphi_i$  with the choice $p_{j,i}=\frac{q_j}{K}$ and $\alpha_0=\left(\sum_{i=1}^K\frac{\varphi_i}{K}\right)\left(\sum_{j=1}^Nq_j\lambda_j\right)$. Moreover, \eqref{eq:domain}, Theorem~\ref{thm: n=1} and the continuity of the spectrum give the second claim by some algebraic manipulation. The proof can be finished by applying Theorem~\ref{thm:irreg2}.
\end{proof}

The difference between the usual Birkhoff averages and weighted Birkhoff averages is shown by the following example:

\begin{ex} \label{ex:alter}
On the full shift system $\Sigma=\{0,1\}^\N$ there exist a potential $\varphi\colon\Sigma\mapsto\R$ depending only on the first symbol and a bounded sequence of weights $\ww=(w_i)_i$ (which is not frequency regular) such that
\begin{itemize}
\item[--] The $\alpha=0$ is the only possible value of the limit of a weighted Birkhoff average,
\item[--] $0<h_{\rm top} (E_\ww(0)) < \log 2$.
At all the points in $\Sigma\setminus E_\ww(0)$ the limit of the weighted Birkhoff average does not exist.
\end{itemize}
\end{ex}

In particular, to have non-degenerate weighted spectrum, the frequency regularity of the weights is somewhat necessary. The proof of the example will be given in the last section.

\subsection*{Structure of paper.} In Section \ref{sec:Topological entropy}, we recall the definition and basic properties of topological entropy. In Section \ref{sec:Continuity and concavity of the spectrum}, we prove Theorem \ref{thm:cont} and Theorem \ref{thm:contgen}. In Section \ref{sec:Typical weights}, we prove Theorem \ref{thm:typmain} and Theorem \ref{cor:main}. In Section \ref{sec:Frequency regular sequences}, as an application of Theorem \ref{thm:cont} and Theorem \ref{thm:typmain}, we show Theorem \ref{thm:goal} and Theorem \ref{thm:irreg2}. We remark that the proof of Theorem \ref{thm:cont} and that of Theorem \ref{thm:typmain} are independent. Thus the readers who are interested in Theorem \ref{thm:typmain} may read directly Section \ref{sec:Typical weights}.

\section{Continuity and concavity of the spectrum}\label{sec:Continuity and concavity of the spectrum}

Let us recall the conditions and notations of Theorem~\ref{thm:cont}. That is, we assume that $\Sigma_{\mathbf{A}}\subseteq\Sigma=\mathcal{A}^\N$ is an aperiodic and irreducible subshift of finite type. Moreover, let $\phi_i:\Sigma_{\mathbf{A}}\to\R$ be a sequence of uniformly equicontinuous potentials.
For $\ii\in \Sigma_{\mathbf{A}}$, let
$$
\overline{A}(\ii):= \limsup_{n\to\infty} \frac 1n \sum_{i=0}^{n-1} \phi_i(\sigma^i \ii)\text{ and }\underline{A}(\ii):= \liminf_{n\to\infty} \frac 1n \sum_{i=0}^{n-1} \phi_i(\sigma^i \ii).
$$

Given $\alpha\leq\beta\in\R$, let
\[
L_{\mathbf{A}}(\alpha,\beta)=\{\ii\in\Sigma: \underline{A}(\ii)=\alpha\text{ and }\overline{A}(\ii)=\beta \}.
\]
For short, let $L_{\mathbf{A}}(\alpha):=L_{\mathbf{A}}(\alpha,\alpha)$. Define
\[
B_m^n(\ii) :=  \sum_{i=m}^{n-1} \phi_i(\sigma^i \ii)
\]
and $A_m^n(\ii)=\frac{1}{n-m}B_m^n(\ii)$. Let
\[
\rho_n^{(2)} := \sup_{\ii\in \Sigma_{\mathbf{A},n}} \sup_{\jj,\mathbf{k}\in[\ii]} |A_0^n(\jj)-A_0^n(\mathbf{k})|,
\]
for $m,n\in \N$ with $n>m$.
It is clear that
\begin{equation}
\rho_n^{(2)}\leq \frac 1n \sum_{i=1}^n \rho_i^{(1)}.
\end{equation}
Since $\rho_n^{(1)}$ converges to $0$ as $n$ tends to $\infty$, so does $\rho_n^{(2)}$.

\begin{lem}\label{lem:m n k}
	Let $\varepsilon>0$ and $N\in \N$. Suppose that $|A_0^n(\ii)-\alpha|<\varepsilon$ for all $n>N$.
	Then for $m,n > N$ we have	
	\[
	|A_m^n(\ii)-\alpha| \leq \varepsilon\frac{n+m}{n-m}.
	\]
\end{lem}
\begin{proof}
	The statement follows simply from $(n-m) A_m^n(\ii) = n A_0^n(\ii) - mA_0^m(\ii).$
\end{proof}

We remind that for an aperiodic and irreducible subshift of finite type $\Sigma_{\mathbf{A}}$ there exists a constant $r$ such that for any two admissible words $\ii, \jj\in\Sigma_{\mathbf{A},*}$ there exists a word $\mathbf{k}$ of length $r$ such that the concatenation $\ii\mathbf{k}\jj$ is admissible, moreover one can choose $\mathbf{k}$ depending only on the last symbol of $\ii$ and the first symbol of $\jj$. We fix $r$ for the rest of the section.

We will need the following technical lemma. Note that although the sequence $\phi_i$ is defined only on $\Sigma_{\mathbf{A}}\subseteq\Sigma$, it can be naturally extended to $\Sigma$ in such a way that the sequence remains uniformly equicontinuous. For instance, for every $\ii\in\Sigma$ let $n(\ii)=\inf\{n\geq0:\ii|_0^n\in\Sigma_{\mathbf{A},*}\}$, that is, $\ii|_0^{n(\ii)}$ is the longest admissible prefix of $\ii$ and let $\phi_i(\ii):=\max_{\jj\in[\ii|_0^{n(\ii)}]}\phi_i(\jj)$. 
We consider a map $\pi$ in the following lemma, which illustrates that the concatenation of a sequence of admissible words can be changed into an admissible infinite sequence without changing the weighted Birkhoff average.

\begin{lem} \label{lem:techn}
Let $(q_j)_{j=1}^\infty$ be an increasing sequence of integers satisfying $q_j/j\to\infty$ and $q_{j+1}-q_j> 2r$. Let $\pi\colon\Sigma\to\Sigma$ be a map satisfying the following properties for every $n\in\N$:
\begin{itemize}
\item[i)] if $\ii|_0^{n}=\jj|_0^{n}$ then $(\pi\ii)|_{0}^{q_j}=(\pi\jj)|_0^{q_j}$ for $j$ such that $q_j<n\leq q_{j+1}$, 
\item[ii)] if $\ii|_n^n\neq (\pi \ii)|_n^n$ then $n\in \{q_j+1,\ldots, q_j+r\}$ for some $j$.
\end{itemize}
Then there exists a sequence $\rho^{(3)}_n\searrow 0$ such that for every $\ii\in \Sigma$ and for every $n$
\[
|A_0^n(\pi\ii)-A_0^n(\ii)| < \rho^{(3)}_n.
\]
Moreover, for every $X\subset \Sigma$
\[
h_{\rm top}(\pi(X))= h_{\rm top}(X).
\]
\end{lem}
\begin{proof}
	Taking $j$ such that $q_j<n\leq q_{j+1}$ we get
\[
\begin{split}
|A_0^{n}(\pi \ii) - A_0^{n}(\ii)|& \leq \frac{(j+1)r}{n} \max_{i\geq0,\ii\in\mathcal{A}^\N}|\phi_i(\ii)| + \frac{1}{n}\sum_{i=1}^j \sum_{\ell=0}^{q_{i}-q_{i-1}-r}\rho^{(1)}_{\ell}+\frac{1}{n}\sum_{\ell=\max\{q_{j+1}-q_j-r-n, 0\}}^{q_{j+1}-q_j-r}\rho_\ell^{(1)},\\
&\leq\frac{(j+1)r}{n} \max_{i\geq0,\ii\in\mathcal{A}^\N}|\phi_i(\ii)| + \frac{1}{n}\sum_{i=1}^j(q_{i}-q_{i-1}-r)\rho_{q_i-q_{i-1}-r}^{(2)} +\rho_n^{(2)}.
\end{split}
\]
Observe that $\frac{1}{n}\sum_{i=1}^j(q_{i}-q_{i-1}-r)\rho_{q_i-q_{i-1}-r}^{(2)}\to0$ as $n\to\infty$. Indeed, since $q_{j}-q_{j-1}-r\to\infty$ as $j\to\infty$, for every $\varepsilon>0$ there exists $J>0$ so that for every $i\geq J$ $\rho_{q_{i}-q_{i-1}-r}^{(2)}<\varepsilon$ and thus, $\frac{1}{n}\sum_{i=1}^j(q_{i}-q_{i-1}-r)\rho_{q_i-q_{i-1}-r}^{(2)}\leq \frac{q_j-q_{J-1}}{n}\varepsilon+\frac{q_J\rho_1^{(2)}}{n}$. This proves the first assertion.

To prove the second assertion, we need a lower and an upper bound. For the upper bound we notice that the image under $\pi$ of a cylinder whose level is not of form $ \{q_j+1,\ldots, q_j+2r\}$ is contained in a cylinder of the same level. As for any set $X$ we can construct a family of covers realizing the topological entropy using only cylinders of levels not of form $ \{q_j+1,\ldots, q_j+2r\}$, the images of those cylinders will give us a family of covers of $\pi(X)$ realizing the same topological entropy.

For the lower bound, let $\mu$ be a measure supported on $X$ such that for every cylinder $C$ of level $\ell(C)=n$ we have
\[
\mu(C\cap X) \leq e^{(h_{\rm top}(X)+\varepsilon)\ell(C)}.
\]
Then if $n$ is not of form $ \{q_j+1,\ldots, q_j+2r\}$ but $q_j<n\leq q_{j+1}$ then for every cylinder $C'$ of level $n$ we have
\[
\pi_*(\mu)(C') \leq K^{(j+1)r} e^{(h_{\rm top}(X)+\varepsilon)\ell(C)}.
\]
Intuitively speaking  but not quite precisely, the map $\pi$ acting on initial words of length $\leq q_{j+1}$ is at most $K^{(j+1)r}$-to-1. As $j=o(q_j)$, the factor $K^{r(j+1)}$ is subexponential in $q_j$ and thus we get the lower bound from Lemma \ref{lem:Forstman}.
\end{proof}

The proof of Theorem~\ref{thm:cont} relies on the following technical proposition. It is a weighted version of the w-measure construction in Gelfert and Rams \cite[Section~5.2]{GR}. In simple but very vague words, we have some collections of sequences with given weighted Birkhoff averages $\alpha_i$ and we concatenate proper parts of them to construct 'Frankenstein' sequences with weighted Birkhoff average $\lim \alpha_i$. Important part is that if our starting collections were large (of large topological entropy), we can do it in a way that the constructed set of sequences also has large topological entropy.

\begin{prop}\label{prop:technical}
	Let $\varepsilon_n>0$ and $\alpha_n$ be sequences of reals such that\linebreak $\limsup_{n\to\infty}\alpha_n=\alpha_{\rm max}$, $\liminf_{n\to\infty}\alpha_n=\alpha_{\rm min}$ and $\lim_{n\to\infty}\varepsilon_n=0$. Moreover, assume that for every $n\geq1$ there exists a set $M_n\subset\Sigma_{\mathbf{A}}$ and a positive integer $T_n>0$ such that for every $\ii\in M_n$ and $m\geq T_n$
	$$
	\left|\frac{1}{m}\sum_{k=0}^{m-1}\phi_k(\sigma^k\ii)-\alpha_n\right|<\varepsilon_n.
	$$
	Then $h_{\rm top}(L_{\mathbf{A}}(\alpha_{\rm min},\alpha_{\rm max}))\geq\liminf_{n\to\infty}h_{\rm top}(M_n)$.
	
	Moreover, in case $\lim_{n\to\infty}\alpha_n=\alpha$ then there exists a set $M\subset L_{\mathbf{A}}(\alpha)$ such that
	the convergence $A_0^n(\ii)\to\alpha$ is uniform on $M$ and $h_{\rm top}(M)\geq\limsup_{n\to\infty}h_{\rm top}(M_n)$.
\end{prop}

For a subset $M\subset \Sigma$, we denote by $M[a,b]=\{\ii\in \AA^{b-a+1}: \exists \jj\in M, \jj|_a^b=\ii \}$. That is, the collection of $(b-a+1)$-words occurring in certain element of $M$ starting at place $a$ and ending at place $b$. Moreover, we use the notation $Z_a^b(M)=\#M[a,b-1]$ for convenience. It is clear that for $a<b<c$, we have $Z_a^c(M) \leq Z_a^b(M) \cdot Z_b^c(M)$.

\begin{lem}\label{lem:1a}
	Let $M$ be a set with $h_{\rm top}(M)>0$. Then for every $h<h_{\rm top}(M)$ there exists a sequence $(z_i)_{i\in \N}$ of $\N$ such that for every $z_i$ and for every $n>z_i$ we have $$\log Z_{z_i}^n(M) > (n-z_i)h.$$
\end{lem}

\begin{proof}
	Indeed, if it fails then we would be able to find  an increasing subsequence $(n_i)_{i\in \N}$ of $\N$  such that $\log Z_{n_i}^{n_{i+1}}(M) \leq (n_{i+1}-n_i)h$, and by summing them up this would imply $\log Z_0^{n_i}(M) \leq (n_i-n_0)h+\log Z_0^{n_0}(M)$ , hence $h_{\rm top}(M) \leq h$, which is a contradiction.
\end{proof}

\begin{proof}[Proof of Proposition~\ref{prop:technical}]
	Let $M_k$ be the sequence of subsets and $T_k$ as in the assumption. Moreover, let $\inf_k h_{\rm top}(M_k)>\delta>0$ be arbitrary but fixed. Then by Lemma~\ref{lem:1a}, for every $k\in\N$ there exists a sequence $(z_i^k)_{i\in\N}$ such that
	\begin{equation}\label{eq:lowertop}
	\log Z_{z_i^k}^n(M_k) > (n-z_i^k)(h_{\rm top}(M_k)-\delta)\text{ for every }n\geq z_{i}^k.
	\end{equation}
	We choose a subsequence $(N_k)_{k\in \N}$ of $\N$ satisfying the following properties:
	\begin{itemize}
		\item $N_0=0$, $N_{k-1}>T_{k}$;
		\item $N_k\in (z_i^{k+1})_{i\in \N}$;
		\item $\lim_{k\to\infty}\frac{N_{k+1}}{\sum_{j=1}^kN_j}=\infty$;
		\item $\log Z_0^n(M_1)\geq n(h_{\rm top}(M_1)-\delta)$ for all $n\geq N_1$.
	\end{itemize}
	Now, let us define sequences $2\geq r_k>1$ and $m(k)\in\N$ such that
	$$
	(r_k)^{m(k)}=\frac{N_{k}}{N_{k-1}}\text{, }\lim_{k\to\infty}r_k=1\text{ and }\lim_{k\to\infty}(r_k-1)\varepsilon_k^{-1}=\infty.
	$$
	
	Define a sequence $(t_i^{k})_{i=0}^{m(k)}$ by $t_{i}^k=\lfloor (r_k)^iN_{k-1} \rfloor$ for $i=0,\ldots,m(k)$. By definition, $t_{m(k)}^k=t_0^{k+1}$. It is easy to check that
	$$r_k -\frac{1}{N_{k-1}}\leq\frac {t_{i+1}^k} {t_{i}^k}\le r_k+\frac{2}{N_{k-1}}\text{ for }1\le i\le m(k)-1.$$

	Finally, let
	\begin{equation*}
	\begin{split}
	\widetilde{M}&=\{\ii\in \mathcal{A}^\N:  \ii|_0^{N_1-1}\in M_1[0, N_1-1], \\
	&\qquad\ii|_{t_i^k}^{t_{i+1}^k-1}\in M_k[t_i^k, t_{i+1}^k-1], \forall 0\le i\le m(k)-1, \forall k\ge 2 \}\\
	&=M_1[0, N_1-1]\times\prod_{k=2}^\infty\prod_{i=0}^{m(k)-1}M_k[t_i^k, t_{i+1}^k-1].
	\end{split}
	\end{equation*}
	In other words,	on positions $0,\ldots, N_1-1$ we can put any sequence that appears in $M_1$. For $k>1$, on positions in each $[t_i^k, t_{i+1}^k-1]$ we can put any sequence that can appear (on those positions) in $M_k$. Note that $\widetilde{M}$ is not necessarily a subset of $\Sigma_{\mathbf{A}}$, since it might happen that these concatenations are forbidden. We will use this set to construct one with the properties claimed in the statement, but first we show that $\widetilde{M}$ is a prototype of our goal set. Namely, we will first show that the set $\widetilde{M}\subseteq\Sigma$ satisfies
	\begin{enumerate}[(i)]
		\item\label{it:inL} $\alpha_{\min}\leq\liminf_{n\to\infty}A_0^n(\ii)\leq \limsup_{n\to\infty}A_0^n(\ii)\leq\alpha_{\max}$ for every $\ii\in\widetilde{M}$,
		\item\label{it:htop} $h_{\rm top}(\widetilde{M})\geq\liminf_{n\to\infty}h_{\rm top}(M_n)$.
	\end{enumerate}
	
	Consider $\ii\in \widetilde{M}$ and $n\in \N$. Take $k\in \N$ with $N_k\leq n < N_{k+1}$. Let $m$ be the largest number such that $n-t_m^{k+1}>0$. Remark that
	$$
	B_0^n(\ii)=B_0^{N_1}(\ii)+\sum_{j=2}^{k}\sum_{\ell=0}^{m(j)-1}B_{t_\ell^j}^{t_{\ell+1}^j}(\sigma^{t_\ell^j}\ii)+\sum_{\ell=0}^{m-1}B_{t_\ell^{k+1}}^{t_{\ell+1}^{k+1}}(\sigma^{t_\ell^{k+1}}\ii)+B_{t_m^{k+1}}^n(\sigma^{t_m^{k+1}}\ii).
	$$
	Observe that for every $t_{\ell}^j$ there exists a $\jj\in M_{j}$ such that for every $t_{\ell}^j\leq i<t_{\ell+1}^j$, $|\phi_i(\sigma^i\ii)-\phi_i(\sigma^i\jj)|\leq\mathrm{var}_{t_{\ell+1}^j-i}(\phi_i)$, and thus
	$$
	B_{t_\ell^j}^{t_{\ell+1}^j}(\sigma^{t_\ell^j}\ii)=\sum_{i=t_\ell^j}^{t_{\ell+1}^j-1}\phi_i(\sigma^i\ii)\leq\sum_{i=t_\ell^j}^{t_{\ell+1}^j-1}\phi_i(\sigma^i\jj)+\sum_{i=t_\ell^j}^{t_{\ell+1}^j-1}\mathrm{var}_{t_{\ell+1}^j-i}(\phi_i).
	$$
	Hence, by Lemma \ref{lem:m n k}
	\[\begin{split}
	B_{0}^n(\ii)\leq&\alpha_1N_1+\sum_{j=2}^{k}\alpha_{j}(N_{j}-N_{j-1})+(n-N_k)\alpha_{k+1}+\varepsilon_1 N_1\\
	&+\sum_{j=2}^{k}\sum_{\ell=0}^{m(j)-1}\varepsilon_{j}(t_{\ell+1}^j+t_\ell^j)+\sum_{\ell=0}^{m-1}\varepsilon_{k+1}(t_{\ell+1}^{k+1}+t_\ell^{k+1})+(n+t_m^{k+1})\varepsilon_{k+1}\\
	&+\sum_{j=2}^{k}\sum_{\ell=0}^{m(j)-1}\sum_{i=t_\ell^j}^{t_{\ell+1}^j-1}\mathrm{var}_{t_{\ell+1}^j-i}(\phi_i)+\sum_{\ell=0}^{m-1}\sum_{i=t_\ell^{k+1}}^{t_{\ell+1}^{k+1}-1}\mathrm{var}_{t_{\ell+1}^{k+1}-i}(\phi_i)+\sum_{i=t_{m}^{k+1}}^{n-1}\mathrm{var}_{t_{m+1}^{k+1}-i}(\phi_i).
	\end{split}
	\]
	Observe that
	$$
	\sum_{\ell=0}^{m(j)-1}\varepsilon_{j}(t_{\ell+1}^j+t_\ell^j)\leq\sum_{\ell=0}^{m(j)-1}\varepsilon_{j}r_j^\ell(r_j+1)N_{j-1}\leq\frac{3\varepsilon_{j}N_{j-1}(r_j^{m(j)}-1)}{r_j-1}\leq\frac{3\varepsilon_{j}N_{j}}{r_j-1}.
	$$
	Hence,
	\[
	\begin{split}
	&\frac{1}{n}\left(\sum_{j=2}^{k}\sum_{\ell=0}^{m(j)-1}\varepsilon_{j}(t_{\ell+1}^j+t_\ell^j)+\sum_{\ell=0}^{m-1}\varepsilon_{k+1}(t_{\ell+1}^{k+1}+t_\ell^{k+1})+(n+t_m^{k+1})\varepsilon_{k+1}\right)\\
	&\qquad\leq\frac{1}{n}\left(\sum_{j=2}^k\frac{3\varepsilon_{j}}{r_j-1}N_{j}+\frac{3\varepsilon_{k+1}N_{k}(r_{k+1}^m-1)}{r_{k+1}-1}+(n+t_m^{k+1})\varepsilon_{k+1}\right)\\
	&\qquad\leq O(\frac{\sum_{j=2}^{k-1}N_j}{N_k})+\frac{3\varepsilon_k}{r_k-1}+\frac{3\varepsilon_{k+1}(t_m^{k+1}+1)}{n(r_{k+1}-1)}+\frac{(n+t_m^{k+1})\varepsilon_{k+1}}{n}\\
	&\qquad\leq O(\frac{\sum_{j=2}^{k-1}N_j}{N_k})+\frac{3\varepsilon_k}{r_k-1}+\frac{6\varepsilon_{k+1}}{r_{k+1}-1}+2\varepsilon_{k+1}=o(1).
	\end{split}
	\]
	On the other hand, since $\rho_i^{(1)}\to0$ as $i\to\infty$, where $\rho_i^{(1)}$ is defined in \eqref{def:rho1}, we get $\frac{1}{i}\sum_{j=1}^{i}\rho_j^{(1)}\to0$ as $i\to\infty$ and hence,
	\begin{multline*}
	\sum_{j=2}^{k}\sum_{\ell=0}^{m(j)-1}\sum_{i=t_\ell^j}^{t_{\ell+1}^j-1}\mathrm{var}_{t_{\ell+1}^j-i}(\phi_i)+\sum_{\ell=0}^{m-1}\sum_{i=t_\ell^{k+1}}^{t_{\ell+1}^{k+1}-1}\mathrm{var}_{t_{\ell+1}^{k+1}-i}(\phi_i)+\sum_{i=t_{m}^{k+1}}^{n-1}\mathrm{var}_{t_{m+1}^{k+1}-i}(\phi_i)\\
	\leq\sum_{j=2}^{k}\sum_{\ell=0}^{m(j)-1}\sum_{i=t_\ell^j}^{t_{\ell+1}^j-1}\rho_{t_{\ell+1}^j-i}^{(1)}+\sum_{\ell=0}^{m-1}\sum_{i=t_\ell^{k+1}}^{t_{\ell+1}^{k+1}-1}\rho_{t_{\ell+1}^{k+1}-i}^{(1)}+\sum_{i=t_{m}^{k+1}}^{n-1}\rho_{t_{m+1}^{k+1}-i}^{(1)}\\
	\leq\sum_{j=2}^{k}\sum_{\ell=0}^{m(j)-1}(t_{\ell+1}^j-t_\ell^j)o(1)+\sum_{\ell=0}^{m-1}(t_{\ell+1}^{k+1}-t_\ell^{k+1})o(1)+(n-t_{m}^{k+1}-1)o(1)\\
	=n\cdot o(1).
	\end{multline*}
	
	The lower bound is similar, thus we get for every $N_k\leq n< N_{k+1}$
	$$
	A_0^n(\ii)=\frac{\alpha_kN_k+(n-N_k)\alpha_{k+1}}{n}+o(1).
	$$
	This together with $N_{k+1}/N_{k}\to\infty$ implies \eqref{it:inL}. Moreover, if $\alpha_k\to\alpha$, this shows that the convergence $A_0^n\to\alpha$ is uniform on $\widetilde{M}$. So it only remains to show \eqref{it:htop}.
	
	
	We pick arbitrarily $t_{\ell}^k\le n<t_{\ell+1}^{k}$ for some $k\in \N$ and $0\le \ell\le m(k)-1$. By definition of $\widetilde{M}$ and \eqref{eq:lowertop},  we have	
	\[
	Z_{N_{k-1}}^{t_\ell^k}(\widetilde{M}) \geq Z_{N_{k-1}}^{t_\ell^k}(M_k) \geq \exp((t_\ell^k-N_{k-1})(h_{\rm top}(M_k)-\delta))
	\]
	The last inequality is due to the fact that $N_{k-1}\in (z_i^k)_{i\in \N}$.
	Similarly, we see that
	$$
	Z_{N_{k-1}}^{N_{k}}(\widetilde{M}) \geq \exp((N_{k}-N_{k-1})(h_{\rm top}(M_k)-\delta)).
	$$
	 Since $\widetilde{M}[N_{i-1}, N_{i}-1]$ $\widetilde{M}[N_{j-1},N_{j}-1]$ are independent for $i\not=j$, we have
	\begin{multline}\label{eq:lowertop2}
	Z_0^{t_\ell^k}(\widetilde{M}) = Z_{N_{k-1}}^{t_\ell^k}(\widetilde{M}) \cdot \prod_{i=1}^{k-1}Z_{N_{i-1}}^{N_{i}}(\widetilde{M})\\
	\ge \exp\left((t_\ell^k-N_{k-1})(h_{\rm top}(M_k)-\delta)+\sum_{i=1}^{k-1}(h_{\rm top}(M_i)-\delta)(N_{i}-N_{i-1})\right)\\
	\geq \exp\left(t_\ell^k(\liminf_{i\to\infty}h_{\rm top}(M_i)-o(1)-\delta)\right).
	\end{multline}
	
	We define a probability measure $\mu$ as follows. For any $\ii\in\Sigma_n$, let $k\in \N$ and $0\le \ell\le m(k)-1$ be the unique integer such that $t_\ell^k< n\leq t_{\ell+1}^k$, and let
	$$
	\mu([\ii])=\frac{\#\{A\in M[0,t_{\ell+1}^k-1]:[\ii]\supset[A]\}}{Z_{0}^{t_{\ell+1}^k}(\widetilde{M})}
	$$
	It is easy to see that $\mu$ is a well defined measure supported on $\widetilde{M}$. Indeed, if $|\ii|<t_{\ell+1}^k$ then
	\[
	\begin{split}
	\sum_{j\in\mathcal{A}}\mu[\ii j]&=\sum_{j\in\mathcal{A}}\frac{\#\{A\in \widetilde{M}[0,t_{\ell+1}^k-1]:[\ii j]\supset[A]\}}{Z_{0}^{t_{\ell+1}^k}(\widetilde{M})}\\
	&=\frac{\#\{A\in \widetilde{M}[0,t_{\ell+1}^k-1]:\text{ there exists $j\in\mathcal{A}$ such that }[\ii j]\supset[A]\}}{Z_{0}^{t_{\ell+1}^k}(\widetilde{M})}\\
	&=\frac{\#\{A\in \widetilde{M}[0,t_{\ell+1}^k-1]:[\ii]\supset[A]\}}{Z_{0}^{t_{\ell+1}^k}(\widetilde{M})},
	\end{split}
	\]
	and if $|\ii|=t_{\ell+1}^k$ then
	\[
	\begin{split}
	\sum_{j\in\mathcal{A}}\mu([\ii j])&=\sum_{j\in\mathcal{A}}\frac{\#\{A\in \widetilde{M}[0,t_{\ell+2}^k-1]:[\ii j]\supset[A]\}}{Z_{0}^{t_{\ell+2}^k}(\widetilde{M})}\\
	&=\frac{\#\{A\in \widetilde{M}[0,t_{\ell+2}^k-1]:\text{ there exists $j\in\mathcal{A}$ such that }[\ii j]\supset[A]\}}{Z_{0}^{t_{\ell+1}^k}(\widetilde{M})Z_{t_{\ell+1}^k}^{t_{\ell+2}^k}(\widetilde{M})}\\
	&=\frac{ Z_{t_{\ell+1}^k}^{t_{\ell+2}^k}(\widetilde{M})\delta_{\ii\in \widetilde{M}[0,t_{\ell+1}^k-1]}}{Z_{0}^{t_{\ell+1}^k}(\widetilde{M})Z_{t_{\ell+1}^k}^{t_{\ell+2}^k}(\widetilde{M})},
	\end{split}
	\]
	where with a slight abuse of notation we used the $t_{m(k)+1}^k:=t_{1}^{k+1}$.
	
	By \eqref{eq:lowertop2}, we have that for every $\ii\in \widetilde{M}$
	\[
	\begin{split}
	\liminf_{n\to\infty} \frac{-\log\mu([\ii|_0^n])}{n}&\ge \liminf_{n\to\infty}\frac{t_\ell^k}{n}\left(\liminf_{i\to\infty}h_{\rm top}(M_i)-o(1)-\delta\right)\\
	&\geq \liminf_{k\to\infty}r_k\left(\liminf_{i\to\infty}h_{\rm top}(M_i)-o(1)-\delta\right)\\
	&=\liminf_{i\to\infty}h_{\rm top}(M_i)-\delta.
	\end{split}
	\]
	By Lemma \ref{lem:Forstman}, we get \eqref{it:htop}.

We are now almost done. We have constructed the set $\widetilde{M}$ which has almost all the demanded properties, the only one that is still missing is that $\widetilde{M}$ is not necessarily a subset of $\Sigma_{\mathbf{A}}$. The last step is to find a map $\pi$ satisfying the assumptions of Lemma \ref{lem:techn} and such that $\pi(\widetilde{M})\subset \Sigma_{\mathbf{A}}$. Observe that the assertion of Lemma \ref{lem:techn} will guarantee that the set $M=\pi(\widetilde{M})$ will satisfy the assertion of Proposition \ref{prop:technical}.

It is easy enough to do. Let us put the points $(t_i^k)_{i,k}$ in the increasing order and denote this sequence by $(q_j)$ (ignoring the initial finitely many terms we can freely assume that $q_{j+1}-q_j>2r$). Observe that each sequence $\ii|_{q_j+1}^{q_{j+1}}$ is an admissible word in $\Sigma_{\mathbf{A},*}$. We can thus modify $\ii$ only on positions $q_j+1,\ldots, q_j+r; j=1, 2, \ldots$ so that we obtain a sequence in $\Sigma_{\mathbf{A}}$. Each modification on positions $q_j+1,\ldots, q_j+r$ can be chosen depending only on $\ii|_{q_{j-1}+1}^{q_j}$ and $\ii|_{q_j+r+1}^{q_{j+1}}$, that is, there exists $\jj\in\Sigma_{\mathbf{A},r}$ such that $\ii|_{q_{j-1}+1}^{q_j}\jj\ii|_{q_j+r+1}^{q_{j+1}}\in\Sigma_{\mathbf{A},*}$. Thus, choosing those modifications in a consistent way we can construct a map $\pi:\Sigma\to\Sigma$ satisfying the assumptions of Lemma \ref{lem:techn} and such that $\pi(\widetilde{M})\subset \Sigma_\mathbf{A}$. 

	
	Finally, to obtain the second part of the assertion let us consider the case when $\lim_{n\to\infty}\alpha_n=\alpha$. By taking a supsequence $n_k$ such that $\limsup_{n\to\infty}h_{\rm top}(M_n)=\lim_{k\to\infty}h_{\rm top}(M_{n_k})$, and applying the previous argument for the sequences $\{\alpha_{n_k}\}_k$ and $\{\varepsilon_{n_k}\}_k$ and $\{M_{n_k}\}_k$ we get the claimed statement.
\end{proof}

\begin{cor}\label{cor:egorov}
	If $L_{\mathbf{A}}(\alpha)\neq\emptyset$ then for every $\delta>0$ there exists $\emptyset\neq M\subset L_{\mathbf{A}}(\alpha)$ such that $h_{\rm top}(M)>h_{\rm top}(L_{\mathbf{A}}(\alpha))-\delta$ and the convergence of $A_0^n(\ii)\to\alpha$ on $M$ is uniform.
\end{cor}

\begin{proof}
	Take a sequence $\varepsilon_n\to0$ be arbitrary but fixed. Since $A_0^n(\ii)\to\alpha$ as $n\to\infty$ for every $\ii\in L(\alpha)$, there exists $N_n(\ii)$ such that for every $m\geq N_n(\ii)$, $|A_0^m(\ii)-\alpha|<\varepsilon_n$. For every $n\geq1$ and $T\geq1$, let
	$$
	M_{n,T}=\{\ii\in L_{\mathbf{A}}(\alpha):N_n(\ii)\leq T\}.
	$$
	Since $L_{\mathbf{A}}(\alpha)=\bigcup_{T=1}^\infty M_{n,T}$ we get that there exists a $T_n$ such that $h_{\rm top}(M_{n,T_n})>h_{\rm top}(L(\alpha))-\delta$.
	
	By applying Proposition~\ref{prop:technical} for the sequence $\alpha_n\equiv\alpha$, $\varepsilon_n$ and $M_n:=M_{n,T_n}$, we get that there exists a set $M\subset L_{\mathbf{A}}(\alpha)$ such that
	$h_{\rm top}(M)\geq\limsup_{n\to\infty}h_{\rm top}(M_{n,T_n})\geq h_{\rm top}(L_{\mathbf{A}}(\alpha))-\delta$, and the convergence is uniform on $M$.
\end{proof}

\begin{cor}\label{cor:lowersemi}
	The map $\alpha\mapsto h_{\rm top}(L_{\mathbf{A}}(\alpha))$ is upper semi-continuous.
\end{cor}

\begin{proof}
	Let $\alpha_n\to \alpha$ be such that $L_{\mathbf{A}}(\alpha_n)\neq\emptyset$. Then we can use Corollary \ref{cor:egorov} to find in each $L_{\mathbf{A}}(\alpha_n)$ a large entropy subset $M_n$ with uniform convergence of the Birkhoff averages, then we apply Proposition \ref{prop:technical} to get the assertion.
\end{proof}

\begin{lem}\label{lem:compact}
	The domain of $\alpha\mapsto h_{\rm top}(L_{\mathbf{A}}(\alpha))$ is compact.
\end{lem}
\begin{proof}
	Suppose $A(\ii_k)=\alpha_k$ for $\ii_k\in \Sigma_{\mathbf{A}}, k\in \N$ satisfying that $\alpha_k\to\alpha$ as $k\to \infty$. We will show $\alpha\in \DD(L_{\mathbf{A}})$. Fix $\epsilon>0$. Then there exists $(N_k)_{k\in \N}$ of positive integers such that for any $n\ge N_k$, $|A_0^n(\ii_k)-\alpha_k|<\epsilon_k$ with $\epsilon_k\to 0$. We pick two sequences $(m_k)_{k\in \N}$ and $(n_k)_{k\in \N}$ satisfying the following conditions:
	\begin{itemize}
		\item $m_0=0$, $m_k<n_k=m_{k+1}-r$ for $k\in \N$.
		\item $n_{k}\ge N_{k+1}$ for $k\in \N$.
		\item $\frac{n_{k-1}}{n_k}\to 0$ as $k\to \infty$.
		\item The set $\{n\in \N: n\notin [m_k, n_k], \forall k  \}$ has density $0$.
	\end{itemize}
	Then we take an $\ii\in \Sigma_{\mathbf{A}}$ such that $\ii|_{m_k}^{n_k}=\ii_k|_{m_k}^{n_k}$ for $k\in \N$ (such $\ii$ exists but may not be unique due to that $\Sigma_{\mathbf{A}}$ is irreducible). It follows from Lemma \ref{lem:m n k} that for $m_k\le n\le n_k$, we have 
	$$
	|A_0^n(\ii)-\alpha|\le \frac{n_{k-1}}{n}|A_0^{n_{k-1}}(\ii)-\alpha|+\frac{n-m_k}{n}|\alpha_k-\alpha|+\frac{r}{n}+2\epsilon_k.
	$$
	Since $\frac{n_{k-1}}{n_k}\to 0$, $\alpha_k\to \alpha$ and $r$ is a constant, we conclude that $A_0^n(\ii)\to \alpha$ as $n\to \infty$.
\end{proof}

The following proposition is in a sense similar to Proposition~\ref{prop:technical}. Like there, we have some given collections of sequences with prescribed weighted Birkhoff averages and use them to construct the large set of their 'Frankenstein' offsprings. However, the technical process of constructing the concatenated sequences is noticeably different.

\begin{prop}\label{prop:concave}
	The domain of $f\colon\alpha\mapsto h_{\rm top}(L_{\mathbf{A}}(\alpha))$ is a (possibly empty) closed convex set and $f$ is a concave function.
\end{prop}

\begin{proof}
	Let $\alpha, \alpha'$ be in the domain of $f$. Assuming that $L_{\mathbf{A}}(\alpha)$ and $L_{\mathbf{A}}(\alpha')$ are nonempty, we want to prove that $L_{\mathbf{A}}(p\alpha + (1-p)\alpha')$ is nonempty and that $f(p\alpha + (1-p)\alpha')\geq pf(\alpha) + (1-p)f(\alpha')$ for all $p\in (0,1)$. Pick arbitrarily $\epsilon>0$.
	By Corollary~\ref{cor:egorov}, for every $\varepsilon>0$ there exist subsets $M(\alpha)\subset L_{\mathbf{A}}(\alpha)$ and $M(\alpha')\subset L_{\mathbf{A}}(\alpha')$ such that
	\begin{itemize}
		\item $h_{\rm top}(M(\alpha))> f(\alpha)-\epsilon$ and $h_{\rm top}(M(\alpha'))> f(\alpha')-\epsilon$;
		\item there exists an increasing sequence $(N_k)_{k\in \N}$ such that for every $\ii\in M(\alpha)$ and every $\ii'\in M(\alpha')$, for every $k$ for every $n>N_k$ we have $|A_0^n(\ii)-\alpha| \leq 1/k$ and $|A_0^n(\ii')-\alpha'| \leq 1/k$.
	\end{itemize}
	We choose two sequences $(t_i)_{i\in \N}, (s_i)_{i\in \N}$ satisfying the following conditions.
	\begin{itemize}
		\item [(i)] $t_0=0$, $t_i\nearrow\infty$ and $t_{i+1}/t_i\searrow 1$.
		\item [(ii)] $s_i\to\infty$.
		\item [(iii)]$(t_{i+1}-t_i)$ is divisible by $s_i$ and $\frac{t_{i+1}-t_i}{s_i}\nearrow\infty$.
		\item[(iv)] $\frac{2s_it_{i+1}}{n(t_{i+1}-t_i)}\to 0$ where $n$ is the largest number such that $N_n<t_i$.
	\end{itemize}
	For example, we can choose $t_{i+1}/t_i \sim 1+n^{-1/2}$ and $s_i \sim n^{1/3}$ where $n$ is the largest number such that $N_n<t_i$.
	We divide each interval $[t_i, t_{i+1}-1]$ into $s_i$ equal subintervals, with endpoints $z_0^i=t_i, z_1^i=t_{i}+(t_{i+1}-t_i)/s_i,\ldots,z_{s_i}^i=t_{i+1}$.
	We will construct a set $\widetilde{M}\subset\Sigma$ step by step as follows.
	
	{\rm Step 0.} At positions $0,\ldots, t_1-1$ we can put anything.
	
	{\rm Step i $(i\ge 1)$.}  We put the $s_i$ numbers
	\[
	W_k^i := \log Z_{z_k^i}^{z_{k+1}^i}(\alpha) - \log Z_{z_k^i}^{z_{k+1}^i}(\alpha'); k=0,1,\ldots s_i-1
	\]
	in an increasing order and we choose $\lfloor p s_i \rfloor$ largest ones. At those chosen intervals the sequences in $\widetilde{M}$ will be taken from $M_{z_k^i}^{z_{k+1}^i}(\alpha)$, at the not chosen intervals from $M_{z_k^i}^{z_{k+1}^i}(\alpha')$.
	
	It is enough to show that $\widetilde{M}\subset\Sigma$ has the following properties:
	
	\medskip
	
	{\rm Claim 1:} for $\ii\in \widetilde{M}$ we have
	\[
	A(\ii) = p\alpha + (1-p)\alpha'.
	\]
	
	\medskip
	
	{\rm Claim 2:} $h_{\rm top}(\widetilde{M})\geq pf(\alpha) + (1-p)f(\alpha')$.
	
	\medskip
	
	Indeed, just like in the proof of Proposition~\ref{prop:technical}, we will prove that there exists a map $\pi\colon\Sigma\mapsto\Sigma$ such that $\pi(\widetilde{M})\subseteq\Sigma_{\mathbf{A}}$ and the assumptions of Lemma~\ref{lem:techn} hold.
	
	\begin{proof}[Proof of Claim 1]
		As $t_{i+1}/t_i\to 1$, it is enough to check that $A_0^{t_i}(\ii)\to p\alpha + (1-p)\alpha'$ as $i$ tends to $\infty$. Pick $i$ and $n$ such that $N_{n}<t_i\le N_{n+1}$.
		By Lemma \ref{lem:m n k}, we have
		\begin{equation*}
		\begin{split}
		&|A_{t_i}^{t_{i+1}}(\ii)-\left(p\alpha + (1-p)\alpha'\right)|\\
		=&|\sum_{k=0}^{s_i-1}\frac{1}{s_i} A_{z_k^i}^{z_{k+1}^i}(\ii)-\left(p\alpha + (1-p)\alpha'\right)|\\
		\le &I_1^i + I_2^i+I_3^i,
		\end{split}
		\end{equation*}
		where
		$$
		I_1^i=\sum_{k=0}^{s_i-1}\frac{1}{s_i} \rho^{(2)}_{\frac{t_{i+1}-t_i}{s_i}}=\rho^{(2)}_{\frac{t_{i+1}-t_i}{s_i}},
		$$
		$$
		I_2^i=\sum_{k=0}^{s_i-1}\frac{1}{s_i} \cdot \frac{z_k^i+z_{k+1}^i}{n\left( \frac{t_{i+1}-t_i}{s_i}\right)}\le \frac{2s_it_{i+1}}{n(t_{i+1}-t_i)},
		$$
		and
		$$
		I_3^i=\left|\frac{1}{s_i}\left(\lfloor ps_i \rfloor \alpha +(s_i-\lfloor ps_i \rfloor) \alpha \right)-\left(p\alpha + (1-p)\alpha'\right)\right|.
		$$
		By $(ii)$, it is easy to see that $I_3^i$ converges $0$ as $i$ tends to $\infty$.
		By $(iii)$ and the fact that $\rho_\ell^{(2)}\to 0$ as $\ell \to \infty$, we see that $I_1^i$ converges $0$ as $i$ tends to $\infty$. By $(iv)$, $I_2^i$ converges $0$ as $i$ tends to $\infty$. Thus we obtain that $A_{t_i}^{t_{i+1}}(\ii)$ converges $p\alpha + (1-p)\alpha'$ as $i$ tends to $\infty$. Since $A_0^{t_i}(\ii)=\frac{1}{t_i}\sum_{j=0}^{i-1} (t_{j+1}-t_j) A_{t_j}^{t_{j+1}}(\ii)$, we complete the proof.
	\end{proof}
	
	\begin{proof}[Proof of Claim 2]
		Observe that the constructions for different $j$ are completely independent from each other: whatever the initial $t_j$ symbols of $\ii\in \widetilde{M}$, we allow any admissible $t_{j+1}-t_j$ symbols to follow.
		Thus we have
		\begin{equation}\label{eq:11}
		Z_0^{t_i}(\widetilde{M})=Z_0^{t_1}(\widetilde{M}) \cdot \prod_{k=1}^{i-1} Z_{t_k}^{t_{k+1}}(\widetilde{M})
		\end{equation}
		and
		\begin{equation}\label{eq:12}
		Z_{t_k}^{t_{k+1}}(\widetilde{M}) \geq \left(\prod_{\ell=0}^{s_k-1} Z_{z_\ell^k}^{z_{\ell+1}^k}(M(\alpha))\right)^{\lfloor ps_k\rfloor/s_k} \cdot \left(\prod_{\ell=0}^{s_k-1} Z_{z_\ell^k}^{z_{\ell+1}^k}(M(\alpha'))\right)^{1-\lfloor ps_k\rfloor/s_k},
		\end{equation}
		for $1\le k\le i-1$.
		Moreover, we have
		\begin{equation}\label{eq:13}
		\prod_{\ell=0}^{s_k-1} Z_{z_\ell^k}^{z_{\ell+1}^k}(M(\alpha)) \geq Z_{t_k}^{t_{k+1}}(M(\alpha))
		\end{equation}
		and
		\begin{equation}\label{eq:14}
		\prod_{k=1}^{i-1} Z_{t_k}^{t_{k+1}}(M(\alpha))| \geq Z_{t_1}^{t_{i}}(M(\alpha)).
		\end{equation}
		The same holds for $\alpha'$.
		We define the probability measure $\mu$ such that for an $\ii\in\Sigma_n$ let $t_{i-1}<n\leq t_{i}$ and
		$$
		\mu([\ii])=\frac{\#\{A\in \widetilde{M}[0,t_i-1]:[\ii]\supset A \}}{Z_{0}^{t_i}(\widetilde{M})}.
		$$
		Similarly to the proof of Proposition~\ref{prop:technical}, $\mu$ is a well defined probability measure supported on $\widetilde{M}$. By \eqref{eq:11}, \eqref{eq:12}, \eqref{eq:13} and \eqref{eq:14}, as $t_{i+1}/t_i\to 1$, we have that
		\begin{equation*}
		\begin{split}
		&\liminf_{n\to\infty} \frac{-\log\mu(C_n\cap \widetilde{M})}{n} \\
		\ge & \liminf_{i\to\infty} \frac 1 {t_i} \left(p\log Z_0^{t_i}(M(\alpha))+(1-p)\log Z_0^{t_i}(M(\alpha')) \right)\\
		\ge & pf(\alpha)+(1-p)f(\alpha'),
		\end{split}	
		\end{equation*}
		for any decreasing sequence $(C_n)_{n\in \N}$ of cylinders with $C_n\cap \widetilde{M}\not=\emptyset$.
		By Lemma \ref{lem:Forstman}, this completes the proof.	
	\end{proof}


	As in the proof of Proposition \ref{prop:technical}, we have now obtained a set $\widetilde{M}$ satisfying all the necessary properties except for one: it does not have to be contained in $\Sigma_{\mathbf{A}}$. Again, we have the same solution to this problem: we will find a map $\pi$ satisfying the assumptions of Lemma \ref{lem:techn} such that $\pi(\widetilde{M})\subset \Sigma_{\mathbf{A}}$. It is done in almost the same manner: we define $(q_j)_j=(z_k^i)_{k,i}$ and then we modify each sequence $\ii\in\widetilde{M}$ on the initial $r$ positions of every interval $(q_j, q_{j+1}]$.
	
	Therefore, we complete the proof.
	
\end{proof}

\begin{proof}[Proof of Theorem~\ref{thm:cont}]
	Since any concave function is clearly lower semi-continuous, Lemma \ref{lem:compact} and Proposition~\ref{prop:concave} together with Corollary~\ref{cor:lowersemi} implies the claim.
\end{proof}

\begin{proof}[Proof of Theorem~\ref{thm:contgen}]

There are two cases. Consider first the simple case:
$h_{\rm top}(L_{\mathbf{A}}(\alpha_0))= h_{\rm top}(\Sigma_{\mathbf{A}})$.

Fix some  $\varepsilon > 0$. We assume that the spectrum domain is larger than one point, hence
by Theorem \ref{thm:cont} we can find a value $\alpha_1$ such that
$h_{\rm top}(L_{\mathbf{A}}(\alpha_1))> h_{\rm top}(L_{\mathbf{A}}(\alpha_0))-\varepsilon$.  By Corollary \ref{cor:egorov} we can find
a set $M_0$ such that $h_{\rm top} (M_0) > h_{\rm top} (L_{\mathbf{A}}(\alpha_0))-\varepsilon$ and that
the convergence $A_0^n(\ii)\to \alpha_0$ is uniform in $M_0$ and we can find a set $M_1$ such that
$h_{\rm top} (M_1) > h_{\rm top} (L_{\mathbf{A}}(\alpha_1))-\varepsilon$ and that the convergence
$A_0^n(\ii)\to \alpha_1$ is uniform in $M_1$. We then apply the Proposition \ref{prop:technical} to the
sequence of sets $M_1, M_0, M_1, M_0, \ldots$, with $\alpha_n$ being $\alpha_0$ or $\alpha_1$
depending on $n$ being even or odd. We get
\[
h_{\rm top}(L_{\mathbf{A}}(\alpha_0, \alpha_1))\geq h_{\rm top} (L_{\mathbf{A}}(\alpha_0))-2\varepsilon.
\]
Naturally, $L_{\mathbf{A}}(\alpha_0, \alpha_1)\subset D$, hence passing with $\varepsilon$ to zero ends
the proof.

The complicated case is when $h_{\rm top}(L_{\mathbf{A}}(\alpha_0))< h_{\rm top}(\Sigma_{\mathbf{A}})$. Note that
we can still freely assume that $h_{\rm top} (\Sigma_{\mathbf{A}}\setminus D) = h_{\rm top}(\Sigma_{\mathbf{A}})$, otherwise
we would have $h_{\rm top}(D) = h_{\rm top}(\Sigma_{\mathbf{A}})$ immediately. We start with a simple observation.

\begin{lem} \label{lem:meeting}
There exists $\beta_0$ such that the sets $\{\ii\in \Sigma_{\mathbf{A}}; A(\ii)< \beta_0\}$ and
$\{\ii\in\Sigma_{\mathbf{A}}; A(\ii)>\beta_0\}$ are both of full entropy $h_{\rm top}(\Sigma_{\mathbf{A}})$.
\end{lem}
\begin{proof}
The function $\beta \to h_{\rm top} (\bigcup_{\alpha<\beta} L_{\mathbf{A}}(\alpha))$ is nondecreasing and
left continuous, hence the set
$\{\beta: h_{\rm top} (\bigcup_{\alpha<\beta} L_{\mathbf{A}}(\alpha)) = h_{\rm top}(\Sigma_{\mathbf{A}})\}$ is closed.
So is the set $\{\beta: h_{\rm top} (\bigcup_{\alpha>\beta} L_{\mathbf{A}}(\alpha)) = h_{\rm top}(\Sigma_{\mathbf{A}})\}$,
for analogous reason. Hence, the two sets must intersect -- otherwise we would have some $\beta$
which would belong to neither, and this is impossible because 
\[
\Sigma_{\mathbf{A}}\setminus D = \bigcup_{\alpha<\beta} L_{\mathbf{A}}(\alpha) \cup \bigcup_{\alpha>\beta} L_{\mathbf{A}}(\alpha) \cup L_{\mathbf{A}}(\beta)
\]
and all three sets on the right would have entropy strictly smaller than the one on the left.
\end{proof}

We fix $\varepsilon>0$. Using again the left-continuity of the function
$\beta \to h_{\rm top} (\bigcup_{\alpha<\beta} L_{\mathbf{A}}(\alpha))$
we can find some $\beta_1<\beta_0$ such that
$ h_{\rm top} (\bigcup_{\alpha<\beta_1} L_{\mathbf{A}}(\alpha)) > h_{\rm top} (\Sigma_{\mathbf{A}}) - \varepsilon$.
Let $M_+ = \bigcup_{\alpha>\beta_0} L_{\mathbf{A}}(\alpha)$ and $M_-= \bigcup_{\alpha<\beta_1} L_{\mathbf{A}}(\alpha)$.

We now need a one-sided version of Proposition \ref{prop:technical}.

\begin{prop} \label{prop:technical2}
Let $\varepsilon_n>0$, $\varepsilon_n\to 0$.
Let $\alpha_n$ be a sequence such that $\alpha_{2k}\to \beta_0$ and $\alpha_{2k+1}\to \beta_1$.
Moreover, assume that for every $n\geq 1$ there exists a set $M_n\subset\Sigma_{\mathbf{A}}$ and a positive
integer $T_n>0$ such that for every $\ii\in M_n$ and $m\geq T_n$ we have
	$$
	\frac{1}{m}\sum_{k=0}^{m-1}\phi_k(\sigma^k\ii)>\alpha_n - \varepsilon_n
	$$
(if $n$ is even) or
$$
	\frac{1}{m}\sum_{k=0}^{m-1}\phi_k(\sigma^k\ii)<\alpha_n + \varepsilon_n
	$$
(if $n$ is odd).
Then we can find a set $M\subset \Sigma_{\mathbf{A}}$ such that for $\ii\in M$ we have
$\underline{A}(\ii) \leq \beta_1$ and $\overline{A}(\ii) \geq \beta_0$ and that
$h_{\rm top} M \geq \liminf h_{\rm top} M_i$.
\end{prop}
\begin{proof}
The proof is virtually identical with the proof of Proposition \ref{prop:technical}.
The construction and the calculation of entropy is the same, the only difference is that
when the sets $M_i$ give only one-sided bounds on the behavior of the Birkhoff sums,
we can only get a weaker statement about $\underline{A}$ and $\overline{A}$. We skip the details.
\end{proof}

We can now fix any sequence $\varepsilon_n\to 0$ and use the sets $M_-$ and $M_+$ defined above to construct the sets $M_n$ satisfying the assumptions of
Proposition \ref{prop:technical2}, in such a way that
$h_{\rm top} M_{2k} > h_{\rm top} M_+ - \varepsilon$ and
$h_{\rm top} M_{2k+1} > h_{\rm top} M_- - \varepsilon$ (by choosing $T_{2k}$, resp. $T_{2k+1}$, large enough).
Using now Proposition \ref{prop:technical2} with those sets $M_n$ we construct a set $M$ which is by
construction contained in $D$, moreover $h_{\rm top} M > h_{\rm top}(\Sigma_{\mathbf{A}})- 2\varepsilon$.
Passing with $\varepsilon$ to 0 we end the proof of this case.
\end{proof}

\section{Typical weights}\label{sec:Typical weights}

First, we need to introduce some notations. Let $\Sigma_{\mathbf{A}}$ be an aperiodic and irreducible subshift of finite type, $\Omega=\Lambda^\N$ and $\Gamma_{\mathbf{A}}=\Omega\times\Sigma_{\mathbf{A}}$. Let $f\colon\Gamma_{\mathbf{A}}\mapsto\R$ be a continuous potential. Let us recall that $S_nf$ denotes the $n$th Birkhoff sum of $f$, that is, $S_nf=f+f\circ\sigma+\cdots+f\circ\sigma^{n-1}$. For every $\ww\in\Omega$ let
$$
Z_n(f,\ww)=\sum_{\ii\in\Sigma_{\mathbf{A},n}}\sup_{\jj\in[\ii]}e^{S_nf(\ww,\jj)},
$$
and define the conditional pressure of $f$ on $\xi(\ww)$ by
\begin{equation}\label{eq:condpresdef}
P(f,\ww)=\limsup_{n\to\infty}\frac{1}{n}\log Z_n(f,\ww).
\end{equation}

The pressure defined in \eqref{eq:condpresdef} corresponds to the definition of the pressure given in Fan \cite[page 3]{F2} in case of $f_k(\ii):=f(\sigma^k\ww,\ii)$ without the extra requirement that it exists as a limit. Later, we will show that for typical $\ww$ with respect to an ergodic quasi-Bernoulli measures $\nu$ it equals to the pressure defined in \eqref{eq:condpresdefdef}.

The following theorem was shown by Ledrappier and Walters \cite{LW}. They proved a more general statement but we state here only the form which corresponds to our main setup.

\begin{thm}[Ledrappier, Walters]\label{thm:LW} Let $\nu$ be a $\sigma$-invariant measure on $\Omega$ and let $f\colon\Gamma_{\mathbf{A}}\mapsto\R$ be a continuous potential. Then
	$$
	\sup\{h_\mu^\xi+\int fd\mu:\mu\in\mathcal{M}_\nu(\Gamma_{\mathbf{A}})\}=\int P(f,\ww)d\nu(\ww).
	$$
\end{thm}

Unfortunately, this theorem itself does not provide enough regularity conditions in order to do multifractal analysis on weighted Birkhoff averages. So we adapt the idea of Ledrappier and Walters \cite{LW} combining with the methods of Takens and Verbitskiy \cite{TV}, Feng \cite{Fe} and Heurteaux \cite{H}.

\subsection{Pinsker's formula} Let us recall that $\Pi$ is the natural projection $\Pi\colon\Omega\times\Sigma\mapsto\Omega$, that is, $\Pi(\ww,\ii)=\ww$. Let $\mu$ be an ergodic $\sigma$-invariant measure on $\Gamma$. Clearly, if $\mu$ is $\sigma$-invariant and ergodic then $\Pi_*\mu$ is $\sigma$-invariant and ergodic on $\Omega$ too. By Shannon-McMillan-Breiman's Theorem,
\begin{equation}\label{eq:shannon}
\begin{split}
h_\mu&=\lim_{n\to \infty}\frac{-1}{n}\log\mu[(\ww,\ii)|_n]\text{ for $\mu$-almost every $(\ww,\ii)$,}\\
h_{\Pi_*\mu}&=\lim_{n\to \infty}\frac{-1}{n}\log\Pi_*\mu[\ww|_n]\text{ for $\Pi_*\mu$-almost every $\ww$.}
\end{split}
\end{equation}

Denote $\xi$ the partition generated by the inverse branches $\Pi^{-1}(\ww)=\{\ww\}\times\Sigma=\xi(\ww)$. By Rohlin's Disintegration Theorem, there exists a family of probability measures $\{\mu_\ww^{\xi}\}$ such that
\begin{enumerate}
	\item $\mu_\ww^{\xi}$ is supported on $\xi(\ww)$;
	\item for every $A\in\mathcal{B}_\Gamma$, the map $\ww\mapsto\mu_\ww^{\xi}(A)$ is $\mathcal{B}_\Omega$-measurable;
	\item $\mu=\int\mu_\ww^{\xi}d\Pi_*\mu(\ww)$.
\end{enumerate}
The family $\{\mu_\ww^{\xi}\}$ of measures is unique up to a zero $\Pi_*\mu$-measure set. Let us define the conditional entropy of $\mu_\ww^\xi$ by
$$
h_{\mu}^\xi:=\int-\log\mu_{\Pi(\ww,\ii)}^\xi([i_0])d\mu(\ww,\ii).
$$
The following theorem is the corresponding version of Pinsker's formula \cite{R}, which we need to establish relation between the conditional entropy, and the entropy of the projection.

\begin{thm}[Pinsker's formula]\label{thm:entconv}
	If $\mu$ is an ergodic $\sigma$-invariant measure then for $\Pi_*\mu$-almost every $\ww$, we have
	\begin{equation}\label{eq:condshannon}
	\lim_{n\to\infty}\frac{-1}{n}\log\mu_\ww^\xi([\ii|_n])=h_{\mu}^\xi\text{ for $\mu_\ww^\xi$-a.e. $\ii$.}
	\end{equation}
	Moreover,
	$$
	h_\mu=h_{\Pi_*\mu}+h_{\mu}^\xi.
	$$
\end{thm}

For completeness, we give a proof here. Observe that the map $(\ww,\ii)\mapsto-\log\mu_{\Pi(\ww,\ii)}^\xi([i_0])$ is in $L^1(\Gamma,\mu)$. Indeed,
\[
\begin{split}
h_{\mu}^\xi&=\int-\log\mu_{\Pi(\ww,\ii)}^\xi([i_0])d\mu(\ww,\ii)\\
&=\int_0^\infty\mu(\{(\ww,\ii):-\log\mu_{\Pi(\ww,\ii)}^\xi([i_0])>x\})dx\\
&=\int_0^\infty\int\ind_{\{-\log\mu_{\Pi(\ww,\ii)}^\xi([i_0])>x\}}(\ww,\ii)d\mu(\ww,\ii)dx\\
&=\sum_{k\in\mathcal{A}}\int_0^\infty\int\ind_{\{-\log\mu_{\Pi(\ww,\ii)}^\xi([i_0])>x\}}(\ww,\ii)\mu^{\xi}_{\Pi(\ww,\ii)}([k])d\mu(\ww,\ii)dx\\
&\leq\sum_{k\in\mathcal{A}}\int_0^\infty\int e^{-x}d\mu(\ww,\ii)dx=K.\\
\end{split}
\]

Let us denote the partition with respect to the cylinders on $\Gamma$ by $\mathfrak{P}$. Then clearly,
\begin{equation}\label{eq:unique}
\sigma_*\left(\mu_{\ww}^{\xi\vee\mathfrak{P}}\right)=\mu_{\sigma\ww}^\xi, ~\text{for }\Pi_*\mu{-a.e. }~ \ww
\end{equation}
Indeed, $\sigma_*\mu_{\ww}^{\xi\vee\mathfrak{P}}$ is supported on $\Pi^{-1}(\sigma\ww)$, and by the definition of conditional measures,
\[
\begin{split}
\int\sigma_*\left(\mu_{\ww}^{\xi\vee\mathfrak{P}}\right)d\Pi_*\mu(\ww)&=\sigma_*\int\left(\mu_{\ww}^{\xi\vee\mathfrak{P}}\right)d\Pi_*\mu(\ww)=\sigma_*\mu=\mu\\
&=\int\mu_{\ww}^{\xi}d\Pi_*\mu(\ww)=\int\mu_{\sigma\ww}^{\xi}d\Pi_*\mu(\ww).
\end{split}
\]
Thus, \eqref{eq:unique} follows by the uniqueness of the conditional measures.

\begin{proof}[Proof of Theorem~\ref{thm:entconv}]
	Let us first show the first assertion of the theorem. By \eqref{eq:unique}, we have
	\[
	\begin{split}
	\mu^\xi_{\Pi(\ww,\ii)}([\ii|_n])&=\mu_{\Pi(\ww,\ii)}^\xi([\ii|_1])\prod_{k=2}^{n}\dfrac{\mu_{\Pi(\ww,\ii)}^\xi([\ii|_{k}])}{\mu_{\Pi(\ww,\ii)}^\xi([\ii|_{k-1}])}\\
	&=\mu_{\Pi(\ww,\ii)}^\xi([\ii|_1])\prod_{k=2}^{n}\mu_{\Pi(\ww,\ii)}^{\xi\vee\mathcal{P}_{k-1}}(\sigma^{-(k-1)}[\sigma^{k-1}\ii|_1])\\
	&=\mu_{\Pi(\ww,\ii)}^\xi([\ii|_1])\prod_{k=2}^{n}\mu_{\Pi\circ\sigma^{k-1}(\ww,\ii)}^{\xi}([\sigma^{k-1}\ii|_1]).
	\end{split}\]
	Taking logarithm and applying Birkhoff's Ergodic Theorem, we get $\frac{-1}{n}\log\mu^\xi_{\Pi(\ww,\ii)}([\ii|_n])=\frac{1}{n}\sum_{k=0}^{n-1}-\log\mu_{\Pi\circ\sigma^{k}(\ww,\ii)}^{\xi}([\sigma^{k}\ii|_1])\to h_\mu^\xi$ for $\mu$-almost every $(\ww,\ii)$. Thus, \eqref{eq:condshannon} follows by Fubini's Theorem.
	
	Now, we show that $h_\mu=h_{\Pi_*\mu}+h_{\mu}^\xi$. By Egorov's Theorem, for every $\varepsilon>0$ there exists $J_1\subset\Gamma$ such that $\mu(J_1)>1-\varepsilon$ and the convergences \eqref{eq:shannon} and \eqref{eq:condshannon} are uniform. That is, there exists $C>0$ such that for every $n\geq1$ and every $(\ww,\ii)\in J_1$
	$$
	C^{-1}e^{-h_{\Pi_*\mu}n}\leq\Pi_*\mu([\ww|_n])\leq Ce^{-h_{\Pi_*\mu}n}\text{ and }C^{-1}e^{-nh_\mu^\xi}\leq\mu_\ww^\xi([\ii|_n])\leq Ce^{-nh_\mu^\xi}.
	$$
	
	By Lebesgue's density Theorem and Egorov's Theorem, there exists $J_2\subset J_1$ such that $\mu(J_2)>1-2\varepsilon$ and there exists $N\geq1$ such that for every $(\ww,\ii)\in J_2$ and $n\geq N$
	$$
	\mu(J_1\cap[(\ww,\ii)|_n])\geq\frac{1}{2}\mu([(\ww,\ii)|_n])\text{ and }\mu_\ww^\xi(J_1\cap[(\ww,\ii)|_n])\geq\frac{1}{2}\mu_\ww^\xi([(\ww,\ii)|_n]).
	$$
	Thus, for every $(\ww,\ii)\in J_2$ and every $n\geq N$
	\[
	\begin{split}
	\mu([(\ww,\ii)|_n])&\leq2\mu(J_1\cap[(\ww,\ii)|_n])\\
	&=2\int\mu_{\Pi(\ww,\ii)}^\xi(J_1\cap[(\ww,\ii)|_n])d\mu(\ww,\ii)\\
	&=2\int_{\Pi^{-1}[\ww|_n]}\mu_{\Pi(\ww,\ii)}^\xi(J_1\cap[(\ww,\ii)|_n])d\mu(\ww,\ii)\\
	&\leq 2\Pi_*\mu([\ww|_n])Ce^{-nh_\mu^\xi}\leq 2C^2e^{-n(h_{\Pi_*\mu}+h_\mu^\xi)}.
	\end{split}
	\]
	On the other hand, for every $(\ww,\ii)\in J_2$
	\[
	\begin{split}
	\mu([(\ww,\ii)|_n])&\geq\mu(J_1\cap[(\ww,\ii)|_n])\\
	&=\int\mu_{\Pi(\ww,\ii)}^\xi(J_1\cap[(\ww,\ii)|_n])d\mu(\ww,\ii)\\
	&=\int_{\Pi^{-1}[\ww|_n]}\mu_{\Pi(\ww,\ii)}^\xi(J_1\cap[(\ww,\ii)|_n])d\mu(\ww,\ii)\\
	&\geq \frac{1}{2}\int_{\Pi^{-1}[\ww|_n]}\mu_{\Pi(\ww,\ii)}^\xi([(\ww,\ii)|_n])d\mu(\ww,\ii)\\
	&\geq \frac{1}{2}\Pi_*\mu([\ww|_n])C^{-1}e^{-nh_\mu^\xi}\geq \frac{1}{2}C^{-2}e^{-n(h_{\Pi_*\mu}+h_\mu^\xi)}.
	\end{split}
	\]
	Thus, the statement follows by Shannon-McMillan-Breiman Theorem.
\end{proof}

\subsection{Regularity of conditional pressure} In this part of the section, we study the regularity properties of the conditional pressure $P(f,\ww)$ under stronger assumptions than the setup of Ledrappier and Walters. Namely, we assume that $f$ has summable variation, that is,
$$
\sum_{k=0}^\infty \max_{(\mathbf{x},\mathbf{k})\in\Gamma_{\mathbf{A},k}}\sup_{(\ww,\ii),(\zz,\jj)\in[(\mathbf{x},\mathbf{k})]}|f(\ww,\ii)-f(\zz,\jj)|<\infty.
$$
Moreover, we assume that the measure $\nu$ is quasi-Bernoulli. Note that for a quasi-Bernoulli measure $\nu$, the transformation $\sigma^m$ is ergodic for every $m\geq1$.

The following lemma is an easy calculation.
\begin{lem}\label{lem:const}
	For every $\ww\in\Omega$,
	$$
	P(f,\ww)=P(f,\sigma\ww).
	$$
	Moreover, if $f\to g$ uniformly then $P(f,\ww)\to P(g,\ww)$.
\end{lem}

\begin{proof}
	Since $f\colon \Gamma_{\mathbf{A}}\mapsto\R$ is continuous over a compact set, we get that $|f|$ is bounded by $C$. Hence,
	\[
	\begin{split}
\sum_{\ii\in\Sigma_{\mathbf{A},n+1}}\sup_{\jj\in[\ii]}e^{S_{n+1}f(\ww,\jj)}&=\sum_{\ii\in\Sigma_{\mathbf{A},n+1}}\sup_{\jj\in[\ii]}e^{S_nf(\sigma\ww,\sigma\jj)}e^{f(\ww,\jj)}\\
&\leq\sum_{\ii\in\Sigma_{\mathbf{A},n+1}}\sup_{\jj\in[\ii]}e^{S_nf(\sigma\ww,\sigma\jj)}e^C\\
&\leq Ke^C\sum_{\ii\in\Sigma_{\mathbf{A},n}}\sup_{\jj\in[\ii]}e^{S_nf(\sigma\ww,\jj)}.\\
	\end{split}
	\]
	The direction $\sum_{\ii\in\Sigma_{\mathbf{A},n+1}}\sup_{\jj\in[\ii]}e^{S_{n+1}f(\ww,\jj)}\geq e^{-C}\sum_{\ii\in\Sigma_{\mathbf{A},n}}\sup_{\jj\in[\ii]}e^{S_nf(\sigma\ww,\jj)}$ is similar.
	
	The second observation follows by the fact that if $\sup_{(\ww,\ii)\in\Gamma_{\mathbf{A}}}|f(\ww,\ii)-g(\ww,\ii)|<\varepsilon$ then $|S_nf-S_ng|\leq \varepsilon n$.
\end{proof}

Since $\nu$ is ergodic, a simple corollary of Lemma~\ref{lem:const} is that we can define the conditional pressure with respect to $\nu$
\begin{equation}\label{eq:condpres2}
P_\nu(f):=\int P(f,\ww)d\nu(\ww)=P(f,\ww)\text{ for $\nu$-almost every }\ww.
\end{equation}
Here, we abused a notation slightly, since $P_\nu(f)$ of \eqref{eq:condpres2} does not necessarily equal to the defined conditional pressure in \eqref{eq:condpresdefdef}. However, we will show in equation \eqref{prop:calcpres} that it is indeed equal to the pressure defined in  \eqref{eq:condpresdefdef}.

For short, for $\ww\in\Omega$ and $\ii\in\Sigma_{\mathbf{A},*}$ let
$$V(f,\ww,\ii):=\sup_{\jj\in[\ii]}e^{S_{|\ii|}f(\ww,\jj)},$$
and for an $\ww\in\Omega_*$ let
$$
Y(f,\ww,\ii):=\sup_{\zz\in[\ww]}V(f,\zz,\ii)\text{ and }W(f,\ww):=\sup_{\zz\in[\ww]}Z_{|\ww|}(f,\zz).
$$
We also use the convention that $Z_{m}(f,\ww)=1$ for $m\leq0$.

Since $f$ has summable variation, there exists constant $C>0$ such that for every $n\geq1$ and every $(\ww,\ii),(\zz,\jj)\in\Gamma_{\mathbf{A}}$ with $|(\ww,\ii)\wedge(\zz,\jj)|=n$
\begin{equation}\label{eq:bd1}
|S_nf(\ww,\ii)-S_nf(\zz,\jj)|<C.
\end{equation}
Thus, for every $\ww\in\Omega$ and every, $\ii,\jj\in\Sigma_{\mathbf{A},*}$ with $\ii\jj\in\Sigma_{\mathbf{A},*}$
\begin{equation}\label{eq:almostmulti}
V(f,\ww,\ii\jj)\leq V(f,\ww,\ii)V(f,\sigma^{|\ii|}\ww,\jj)\leq e^C\cdot V(f,\ww,\ii\jj).
\end{equation}
So clearly, for every $\ww\in\Omega$
\begin{equation}\label{eq:submulti}
Z_{n+m}(f,\ww)\leq Z_{n}(f,\ww)Z_m(f,\sigma^n\ww).
\end{equation}

On the other hand,
\begin{equation}\label{eq:superadditive}
\begin{split}
Z_{n}(f,\ww)Z_m(f,\sigma^n\ww)&\leq K^r e^{r|f|} Z_{n}(f,\ww)Z_{m-r}(f,\sigma^{n+r}\ww)\\
&\leq K^re^{2r|f|+2C} Z_{n+m}(f,\ww),
\end{split}
\end{equation}
 where $r\geq1$ is such that $\mathbf{A}^r$ is strictly positive.

Applying the bounded distortion again, we get for every $(\ww,\ii)\in\Gamma_{\mathbf{A},*}$,  and every $\zz\in[\ww]$ that
\begin{equation}\label{eq:compareYV}
V(f,\zz,\ii)\leq Y(f,\ww,\ii)\leq e^C V(f,\zz,\ii)
\end{equation}
and therefore
\begin{equation}\label{eq:compareZW}
Z_{|\ww|}(f,\zz)\leq W(f,\ww)\leq e^C Z_{|\ww|}(f,\zz).
\end{equation}

By \eqref{eq:submulti} and Kingman's subadditive ergodic theorem, we have that for $\nu$-almost every $\ww\in\Omega$ the limit
$$
\lim_{n\to\infty}\frac{1}{n}\log Z_n(f,\ww)=P(f,\ww)=P_\nu(f)
$$
exists and
\begin{equation}\label{prop:calcpres}
\begin{split}
P_\nu(f)&=\lim_{n\to\infty}\frac{1}{n}\int\log Z_n(f,\ww)d\nu(\ww)\\
&=\lim_{n\to\infty}\frac{1}{n}\sum_{\ww\in\Omega_n}\nu([\ww])\log W(f,\ww),
\end{split}
\end{equation}
where in the last equation we used \eqref{eq:compareZW} too. 

The next theorem is adapting the result and method of Feng \cite[Section~4]{Fe} for the situation of subshift of finite type and to the condition on the marginal measures.

\begin{thm}\label{thm:equillibrium}
	Let $\nu$ be an ergodic $\sigma$-invariant quasi-Bernoulli measure on $\Omega$ and let $f\colon\Gamma_{\mathbf{A}}\mapsto\R$ be a continuous potential with summable variation. Then there exists a unique ergodic $\sigma$-invariant measure $\mu$ such that there exists a constant $C>0$ such that for every $(\ww,\ii)\in\Gamma_{\mathbf{A},*}$
	\begin{equation}\label{eq:comparemu}
	C^{-1}\frac{Y(f,\ww,\ii)}{W_{|\ww|}(f,\ww)}\nu([\ww])\leq\mu([\ww,\ii])\leq C\frac{Y(f,\ww,\ii)}{W_{|\ww|}(f,\ww)}\nu([\ww]).
	\end{equation}
	In particular, $\Pi_*\mu=\nu$ and
	$$
	h_\mu^\xi+\int fd\mu=P_\nu(f).
	$$
\end{thm}

\begin{proof}	
	Let $\zz$ be a generic point such that $\frac{1}{n}\sum_{k=0}^{n-1}\delta_{\sigma^k\zz}\to\nu$ as $n\to\infty$. Then let
	$$
	\eta_m=Z_m(f,\zz)^{-1}\sum_{\ii\in\Sigma_{\mathbf{A},m}}V(f,\zz,\ii)\delta_{(\zz,\ii\jj)},
	$$
	where $\jj\in\Sigma_{\mathbf{A}}$ is arbitrary but fixed. Moreover, let
	$$
	\nu_{n}=\frac{1}{n}\sum_{k=0}^{n-1}\eta_{2n}\circ\sigma^{-k}.
	$$
	Let $\{n_j\}$ be a subsequence such that\linebreak $\lim_{j\to\infty}\frac{1}{n_j}\log Z_{n_j}(f,\ww)=P_\nu^f$ and $\nu_{n_j}\to\mu$. Clearly, $\mu$ is a $\sigma$-invariant measure on $\Gamma_{\mathbf{A}}$.
	
	Fix $(\ww,\ii)\in\Gamma_{\mathbf{A},*}$ with $|\ww|=|\ii|$. Choose $n$ sufficiently large such that $n>|\ww|=|\ii|$. Then by \eqref{eq:almostmulti} and \eqref{eq:superadditive} there exists $C'>0$ such that
	\[
	\begin{split}
	\nu_{n}([\ww,\ii])&=\frac{1}{n}\sum_{k=0}^{n-1}\sum_{\substack{(\alpha,\beta)\in\Gamma_{\mathbf{A},k},(\gamma,\tau)\in\Gamma_{\mathbf{A},2n-|\ii|-k}:\\(\alpha\ww\gamma,\beta\ii\tau)\in\Gamma_{\mathbf{A},2n}
	}}\eta_{2n}([(\alpha\ww\gamma,\beta\ii\tau)])\\
	&=\frac{1}{n}\sum_{k=0}^{n-1}\sum_{\substack{\beta\in\Sigma_{\mathbf{A},k},\tau\in\Sigma_{\mathbf{A},2n-|\ii|-k}:\\ \beta\ii\tau\in\Sigma_{\mathbf{A},2n}}}\frac{V(f,\zz,\beta\ii\tau))}{Z_{2n}(\zz)}\ind_{[\ww]}(\sigma^k\zz)\\
	&\leq\frac{C'}{n}\sum_{k=0}^{n-1}\sum_{\substack{\beta\in\Sigma_{\mathbf{A},k},\\ \tau\in\Sigma_{\mathbf{A},2n-|\ii|-k}}}\frac{V(f,\zz,\beta)V(f,\sigma^k\zz,\ii)V(f,\sigma^{|\ww|+k}\zz,\tau)}{Z_{k}(\zz)Z_{|\ii|}(\sigma^k\zz)Z_{2n-k-|\ii|}(\sigma^{|\ii|+k}\zz)}\ind_{[\ww]}(\sigma^k\zz)\\
	&=\frac{C'}{n}\sum_{k=0}^{n-1}\frac{V(f,\sigma^k\zz,\ii)}{Z_{|\ii|}(\sigma^k\zz)}\ind_{[\ww]}(\sigma^k\zz).
	\end{split}
	\]
	Thus, by Birkhoff's ergodic theorem
	\[
	\begin{split}
	\mu([\ww,\ii])&=\lim_{j\to\infty}\nu_{n_j}([\ww,\jj])\\
	&\leq C'\int\frac{V(f,\zz,\ii)}{Z_{|\ii|}(\zz)}\ind_{[\ww]}(\zz)d\nu(\zz)\\
	&\leq C''\frac{Y(f,\ww,\ii)}{W_{|\ww|}(f,\ww)}\nu([\ww]),
	\end{split}
	\]
	where we used \eqref{eq:compareYV} and \eqref{eq:compareZW}.
	
	Now, we show the other inequality. Similarly by using \eqref{eq:almostmulti}, \eqref{eq:submulti}, we have
	\[
	\begin{split}
	&\nu_{n}([\ww,\ii])\\
	&\geq\frac{1}{n}\sum_{k=0}^{n-1}\sum_{\substack{\beta\in\Sigma_{\mathbf{A},k},\tau\in\Sigma_{\mathbf{A},2n-|\ii|-k}:\\ \beta\ii\tau\in\Sigma_{\mathbf{A},2n}}}\frac{V(f,\zz,\beta)V(f,\sigma^k\zz,\ii)V(f,\sigma^{|\ww|+k}\zz,\tau)}{Z_{k}(\zz)Z_{|\ii|}(\sigma^k\zz)Z_{2n-k-|\ii|}(\sigma^{|\ii|+k}\zz)}\ind_{[\ww]}(\sigma^k\zz)\\
	&\geq\frac{e^{-2|f|r}}{n}\sum_{k=0}^{n-1}\sum_{\substack{\beta'\in\Sigma_{\mathbf{A},k-r}\\\tau'\in\Sigma_{\mathbf{A},2n-|\ii|-k-r}}}\frac{V(f,\zz,\beta')V(f,\sigma^k\zz,\ii)V(f,\sigma^{|\ww|+k+r}\zz,\tau')}{Z_{k}(\zz)Z_{|\ii|}(\sigma^k\zz)Z_{2n-k-|\ii|}(\sigma^{|\ii|+k}\zz)}\ind_{[\ww]}(\sigma^k\zz)\\
	&\geq\frac{e^{-2|f|r}}{n}\sum_{k=0}^{n-1}\frac{V(f,\sigma^k\zz,\ii)}{Z_{|\ii|}(\sigma^k\zz)Z_r(\zz)Z_r(\sigma^{|\ii|+k}\zz)}\ind_{[\ww]}(\sigma^k\zz)\\
	&\geq\frac{e^{-4|f|rK^{-22r}}}{n}\sum_{k=0}^{n-1}\frac{V(f,\sigma^k\zz,\ii)}{Z_{|\ii|}(\sigma^k\zz)}\ind_{[\ww]}(\sigma^k\zz)
	\end{split}
	\]
	and thus, taking the subsequence $n_j$ and using \eqref{eq:compareYV} and \eqref{eq:compareZW}, we have
	$$
	\mu([\ww,\ii])\geq C'^{-1}\frac{Y(f,\ww,\ii)}{W_{|\ww|}(f,\ww)}\nu([\ww]).
	$$
	Now, since $\nu$ is quasi-Bernoulli, by \eqref{eq:almostmulti}-\eqref{eq:submulti} and \eqref{eq:compareYV}-\eqref{eq:compareZW} we have
	\[
	\begin{split}
	\mu([(\ww\mathbf{x},\ii\jj)])&\geq C'^{-1}\frac{Y(f,\ww\mathbf{x},\ii\jj)}{W_{|\ww\mathbf{x}|}(f,\ww\mathbf{x})}\nu([\ww\mathbf{x}])\\
	&\geq C'^{-2}\frac{Y(f,\ww,\ii)}{W_{|\ww|}(f,\ww)}\nu([\ww])\frac{Y(f,\mathbf{x},\jj)}{W_{|\mathbf{x}|}(f,\mathbf{x})}\nu([\mathbf{x}])\\
	&\geq C'^{-4}\mu([(\ww,\ii)])\mu([(\mathbf{x},\jj)]).
	\end{split}
	\]
	This implies that $\mu$ is ergodic. Since $\mu$ was an arbitrary accumulation point and two equivalent ergodic measures are equal, we get that $\mu$ is unique.
	
	For every $\ww\in\Omega_*$, and every $\zz\in[\ww]$
	\[
	\begin{split}
	\Pi_*\mu([\ww])&=\sum_{\ii\in\Sigma_{\mathbf{A},|\ww|}}\mu([\ww,\ii])\\
	&\leq C\sum_{\ii\in\Sigma_{\mathbf{A},|\ww|}}\frac{Y(f,\ww,\ii)}{W_{|\ww|}(f,\ww)}\nu([\ww])\\
	&\leq C^3\sum_{\ii\in\Sigma_{\mathbf{A},|\ww|}}\frac{V(f,\zz,\ii)}{Z_{|\ww|}(f,\zz)}\nu([\ww])\\
	&= C^3\nu([\ww]).
	\end{split}
	\]
	The other inequality $\Pi_*\mu([\ww])\geq C^{-3}\nu([\ww])$ is similar. Since $\Pi_*\mu$ and $\nu$ are both ergodic, we have $\Pi_*\mu=\nu$.
	
	Finally, by \eqref{prop:calcpres}
	\[
	\begin{split}
	h_\mu&=\lim_{n\to\infty}\frac{1}{n}\sum_{(\ww,\ii)\in\Gamma_{\mathbf{A},n}}\mu([\ww,\ii])\log\mu([\ww,\ii])\\
	&=\lim_{n\to\infty}\frac{-1}{n}\sum_{(\ww,\ii)\in\Gamma_{\mathbf{A},n}}\mu([\ww,\ii])\log\left(\frac{Y(f,\ww,\ii)}{W_{|\ww|}(f,\ww)}\nu([\ww])\right)\\
	&=h_\nu-\int fd\mu+P_\nu(f).
	\end{split}\]
	By Theorem~\ref{thm:entconv}, $h_\mu^\xi=h_\mu-h_\nu$, which proves the statement.
\end{proof}	

The next theorem is a modification of the argument of Heurteaux \cite{H}.

\begin{thm}\label{thm:differentiable}
	Let $\nu$ be a $\sigma$-invariant ergodic quasi-Bernoulli measure on $\Omega$ and let $f,g\colon\Gamma_{\mathbf{A}}\mapsto\R$ be a continuous potentials with summable variation. Then the function $p\colon t\mapsto P((1-t)g+tf)$ is differentiable at $t=0$. In particular,
	$$
	p'(0)=\int(f-g)d\mu_g.
	$$
\end{thm}

\begin{proof}
	It is clear by the bounded distortion \eqref{eq:bd1} that there exists a constant $C>0$ such that for every $t\in\R$ and every $(\ww,\ii)\in\Gamma_{\mathbf{A},*}$
	$$
	C^{-1}Y(tf+(1-t)g,\ww,\ii)\leq Y(f,\ww,\ii)^tY(g,\ww,\ii)^{1-t}\leq CY(tf+(1-t)g,\ww,\ii).
	$$
	
	Let $\mu_f$ and $\mu_g$ be the unique ergodic measures defined in Theorem~\ref{thm:equillibrium}. Then for every $t\in\R$ and every $(\ww,\ii)\in\Gamma_{\mathbf{A},*}$
	\[
	\begin{split}
C^{-2}\frac{Y(tf+(1-t)g,\ww,\ii)}{W(f,\ww)^tW(g,\ww)^{1-t}}\nu(\ww)&\leq\mu_f([\ww,\ii])^t\mu_g([\ww,\ii])^{1-t}\\
&\leq C^2\frac{Y(tf+(1-t)g,\ww,\ii)}{W(f,\ww)^tW(g,\ww)^{1-t}}\nu(\ww).
	\end{split}
	\]
	Hence,
	\begin{multline*}
	P_\nu((1-t)g+ft)\\
	=(1-t)P_\nu(g)+tP_\nu(f)+\lim_{n\to\infty}\frac{1}{n}\sum_{\ww\in\Omega_n}\nu([\ww])\log\sum_{\ii\in\Sigma_{\mathbf{A},n}}\frac{\mu_f([\ww,\ii])^t\mu_g([\ww,\ii])^{1-t}}{\nu([\ww])}.
	\end{multline*}
	Thus, it is enough to show that
	$$
	H(t)=\lim_{n\to\infty}\frac{1}{n}\sum_{\ww\in\Omega_n}\nu([\ww])\log\sum_{\ii\in\Sigma_{\mathbf{A},n}}\frac{\mu_f([\ww,\ii])^t\mu_g([\ww,\ii])^{1-t}}{\nu([\ww])}
	$$
	is differentiable.
	
	\medskip
	
	{\rm Claim:}
	There exists a constant $C>0$ such that the sequence
	$$
	\overline{H}_n(t)=\sum_{\ww\in\Omega_n}\nu([\ww])\log\sum_{\ii\in\Sigma_{\mathbf{A},n}}\frac{C\mu_f([\ww,\ii])^t\mu_g([\ww,\ii])^{1-t}}{\nu([\ww])}
	$$
	is submultiplicative $\overline{H}_{n+m}(t)\leq \overline{H}_n(t)+\overline{H}_m(t)$ and
	$$
	\underline{H}_n(t)=\sum_{\ww\in\Omega_n}\nu([\ww])\log\sum_{\ii\in\Sigma_{\mathbf{A},n}}\frac{C^{-1}\mu_f([\ww,\ii])^t\mu_g([\ww,\ii])^{1-t}}{\nu([\ww])}
	$$
	is supermultiplicative $\underline{H}_{n+m}(t)\geq \underline{H}_n(t)+\underline{H}_m(t)$.
	
	\medskip
	
	\begin{proof}[Proof of the Claim]
		By Theorem~\ref{thm:equillibrium} and the equations \eqref{eq:almostmulti}-\eqref{eq:compareZW}, we have that the measures $\mu_f$ and $\mu_g$ are quasi-Bernoulli, and hence, there exists a constant $C>0$ such that
		\begin{multline*}
		\sum_{\ii\jj\in\Sigma_{\mathbf{A},n+m}}\mu_f([\ww,\ii\jj])^t\mu_g([\ww,\ii\jj])^{1-t}\\
		\leq C\sum_{\ii\jj\in\Sigma_{\mathbf{A},n+m}}\mu_f([\ww|_n,\ii])^t\mu_f([\sigma^n\ww,\jj])^t\mu_g([\ww|_n,\ii])^t\mu_g([\sigma^n\ww,\jj])^t\\
		\leq C\sum_{\substack{\ii\in\Sigma_{\mathbf{A},n}\\\jj\in\Sigma_{\mathbf{A},m}}}\mu_f([\ww|_n,\ii])^t\mu_f([\sigma^n\ww,\jj])^t\mu_g([\ww|_n,\ii])^t\mu_g([\sigma^n\ww,\jj])^t.
		\end{multline*}
		On the other hand,
		\begin{multline*}
		\sum_{\ii\jj\in\Sigma_{\mathbf{A},n+m}}\mu_f([\ww,\ii\jj])^t\mu_g([\ww,\ii\jj])^{1-t}\\
		\geq C^{-1}\sum_{\ii\jj\in\Sigma_{\mathbf{A},n+m}}\mu_f([\ww|_n,\ii])^t\mu_f([\sigma^n\ww,\jj])^t\mu_g([\ww|_n,\ii])^t\mu_g([\sigma^n\ww,\jj])^t\\ C^{-1}C'\sum_{\substack{\ii\in\Sigma_{\mathbf{A},n-r}\\\jj\in\Sigma_{\mathbf{A},m-r}}}\mu_f([\ww|_{n-r},\ii])^t\mu_f([\sigma^{n+2r}\ww,\jj])^t\mu_g([\ww|_{n-r},\ii])^{1-t}\mu_g([\sigma^{n+2r}\ww,\jj])^{1-t}\\
		\geq C^{-1}C'K^{-2r}\sum_{\substack{\ii\in\Sigma_{\mathbf{A},n}\\\jj\in\Sigma_{\mathbf{A},m}}}\mu_f([\ww|_{n},\ii])^t\mu_f([\sigma^{n}\ww,\jj])^t\mu_g([\ww|_{n},\ii])^{1-t}\mu_g([\sigma^{n}\ww,\jj])^{1-t}.\\
		\end{multline*}
		\end{proof}
	
	Since $H(0)=0$ and $\overline{H}_n(t)$ is differentiable for every $n$, we get for every $n\geq 1$
	\[
	\begin{split}
	\limsup_{t\to0}\frac{H(t)}{t}&\leq\limsup_{t\to0}\frac{\overline{H}_n(t)}{nt}\\
	&=\left.\frac{1}{n}\sum_{\ww\in\Omega_n}\nu([\ww])\dfrac{\sum_{\ii\in\Sigma_{\mathbf{A},n}}\frac{\mu_f([\ww,\ii])^t\mu_g([\ww,\ii])^{1-t}(\log\mu_f([\ww,\ii])-\log\mu_g([\ww,\ii]))}{\nu([\ww])}}{\sum_{\ii\in\mathcal{A}^n}\frac{\mu_f([\ww,\ii])^t\mu_g([\ww,\ii])^{1-t}}{\nu([\ww])}}\right|_{t=0}\\
	&=\frac{1}{n}\sum_{\ww\in\Omega_n}\nu([\ww])\sum_{\ii\in\Sigma_{\mathbf{A},n}}\frac{\mu_g([\ww,\ii])(\log\mu_f([\ww,\ii])-\log\mu_g([\ww,\ii]))}{\nu([\ww])}\\
	&=\frac{1}{n}\sum_{\substack{\ww\in\Omega_n\\ \ii\in\Sigma_{\mathbf{A},n}}}\mu_g([\ww,\ii])(\log\mu_f([\ww,\ii])-\log\mu_g([\ww,\ii]))\\
	&\leq\frac{C}{n}+\frac{1}{n}\sum_{\substack{\ww\in\Omega_n\\ \ii\in\Sigma_{\mathbf{A},n}}}\mu_g([\ww,\ii])(\log\frac{Y(f,\ww,\ii)\nu([\ww])}{W(f,\ww)}-\log\mu_g([\ww,\ii]))\\
	&\to \int fd\mu_g-h_\nu-P_\nu(f)+h_{\mu_g}\text{ as }n\to\infty,
	\end{split}
	\]
	where we applied again Theorem~\ref{thm:equillibrium}. The other inequality,
	$$
	\liminf_{t\to0}\frac{H(t)}{t}\geq \int fd\mu_g-h_\nu-P_\nu(f)+h_{\mu_g}\text{ as }n\to\infty
	$$
	is similar. Hence,
	$$
	p'(0)=-P_\nu(g)+P_\nu(f)+\int fd\mu_g-h_\nu-P_\nu(f)+h_{\mu_g}=\int fd\mu_g-\int gd\mu_g.
	$$
\end{proof}
	
	\subsection{Weighted Birkhoff average}
For $\alpha,\p\in\R^d$, let us consider the potential $f_{\underline{p}}\colon\Gamma_{\mathbf{A}}\mapsto\R$ defined as
$$
f_{\p}:=\langle \p,f-\alpha\rangle.
$$

First, we show the upper bound in Theorem~\ref{thm:typmain}.

\begin{lem}\label{lem:ub}
	For every $\ww\in\Omega$ and $\alpha\in\R^d$
	$$
		h_{\rm top}(E_\ww(\alpha))\leq\inf_{\p\in\R^d}P(f_{\p},\ww).
	$$
\end{lem}

\begin{proof}
	The proof is standard, but for completeness, we give it here.
	
	Let $s>s_0>\inf_{\p\in\R^d}P(f_{\p},\ww)$. Hence, there exists $\p\in\R^d$ such that $s_0>P(f_{\p},\ww)$. Thus there exists $N'\geq1$ such that for every $n\geq N'$
	$$
	\sum_{\ii\in\Sigma_{\mathbf{A},n}}e^{\langle\p,S_nf-n\alpha\rangle}<e^{s_0n}.
	$$
	By definition,
	\begin{equation}\label{eq:defE}
	E_\ww(\alpha)=\bigcap_{M=1}^\infty\bigcup_{N=1}^\infty\bigcap_{n\geq N}\left\{\ii\in X:\left|\frac{1}{n}S_nf(\ww,\ii)-\alpha\right|<\frac{1}{M}\right\}.
	\end{equation}
	Since $f(\ww,\cdot)\colon \Sigma_{\mathbf{A}}\mapsto\R^d$ is continuous over a compact set, we get that it is uniformly continuous. Thus, for every $M\geq1$ there exists $C>0$ such that for every $n\geq1$, $\ii\in\Sigma_{\mathbf{A},n}$ and every $\jj\in[\ii]$
	$$
	\left|S_nf(\ww,\jj)-\sup_{\jj\in[\ii]}S_nf(\ww,\jj)\right|\leq\frac{Cn}{M}.
	$$
	Choose $M\geq1$ such that $|\p|\frac{1+C}{M}<(s-s_0)/2$. By \eqref{eq:defE}, we get that for every $N$ sufficiently large
\[
\begin{split}
\mathcal{H}_N^s(E_\ww(\alpha))&\leq\sum_{n=N}^\infty\sum_{\substack{\ii\in\Sigma_{\mathbf{A},n} \\ \left|\sup_{\jj\in[\ii]}S_nf(\ww,\jj)-n\alpha\right|<(1+C)n/M}}e^{-n s}\\
&\leq\sum_{n=N}^\infty e^{-n(s-s_0)/2}\sum_{\substack{\ii\in\Sigma_{\mathbf{A},n} \\ \left|\sup_{\jj\in[\ii]}S_nf(\ww,\jj)-n\alpha\right|<(1+C)n/M}}e^{-n s_0+\langle\p,S_nf-n\alpha\rangle}\\
&\leq\sum_{n=N}^\infty e^{-n(s-s_0)/2}\to0\text{ as }N\to\infty.
\end{split}
\]
\end{proof}

Recall that \begin{equation}\label{def:pa}
\mathcal{P}_{\mathbf{A}}=\{\alpha\in\R^d:\text{ there exists }\mu\in\mathcal{M}_\nu(\Gamma_{\mathbf{A}})\text{ such that }\int fd\mu=\alpha\}.
\end{equation}
It is easy to see that $\mathcal{P}_{\mathbf{A}}$ is a closed and convex set. Moreover, without loss of generality, we may assume that $\mathcal{P}_{\mathbf{A}}$ has an interior point. Indeed, if $\mathcal{P}_{\mathbf{A}}$ does not contain interior point then there exists a $d'$-dimensional hyperplane $V$ such that $\mathcal{P}_{\mathbf{A}}\subset V$. By changing coordinates, we may assume that $f\colon\Gamma_{\mathbf{A}}\mapsto\R^{d'}$. Also, for $\nu$-almost every $\ww$,
$$
\mathcal{P}_{\mathbf{A}}=\{\alpha\in\R^d:\text{ there exists }\ii\in\Sigma_{\mathbf{A}}\text{ such that }\lim_{n\to\infty}\frac{1}{n}S_nf(\ww,\ii)=\alpha\}.
$$
Indeed, take the sequence $\mu_n=\frac{1}{n}\sum_{k=0}^n\delta_{\sigma^k\ww,\sigma^k\ii}$ and let $\mu$ be an accumulation point of the sequence $\mu_n$ in the weak*-topology, we get $\int fd\mu=\lim_{k\to\infty}\int fd\mu_{n_k}=\alpha$ and for every $g\in L^1(\Omega)$, $\int gd\Pi_*\mu=\lim_{k\to\infty}\int g\circ\Pi d\mu_{n_k}=\lim_{k\to\infty}\frac{1}{n_k} \sum_{\ell=0}^{n_k}g(\sigma^\ell\ww)=\int gd\nu$. Moreover, since $\sigma_*\mu_n=\mu_n-\frac{1}{n}\delta_{\ww,\ii}+\frac{1}{n}\delta_{\sigma^{n+1}\ww,\sigma^{n+1}\ii}$, we get that $\mu$ is $\sigma$-invariant.

Theorem~\ref{thm:equillibrium} implies that for every $\p\in\R^d$ there exists a $\sigma$-invariant ergodic measure $\mu_{\p}$ such that
\begin{equation}
P_\nu(f_{\p})=h_{\mu_{\p}}^\xi+\int f_{\p}d\mu_{\p}.
\end{equation}

\begin{lem}\label{lem:convex}
The conditional pressure $\p\mapsto P_\nu(f_{\p})$ is convex.
\end{lem}

\begin{proof}
Let $\beta_1,\beta_2>0$ be with $\beta_1+\beta_2=1$ and $\p_1,\p_2\in\R^d$. Then there exist a measure $\mu=\mu_{\beta_1\p_1+\beta_2\p_2}\in\mathcal{E}_\nu(\Gamma_{\mathbf{A}})$ such that
\[
\begin{split}
P_\nu(f_{\beta_1\p_1+\beta_2\p_2})&= h_\mu^\xi+\int f_{\beta_1\p_1+\beta_2\p_2}d\mu\\
&=\beta_1h_\mu^\xi+\beta_2h_\mu^\xi+\beta_1\int f_{\p_1}d\mu+\beta_2\int f_{\p_2}d\mu\\
&\leq \beta_1P_\nu(f_{\p_1})+\beta_2P_\nu(f_{\p_2}).
\end{split}
\]
\end{proof}

\begin{lem}
	For every $\alpha\in\mathcal{P}^o_{\mathbf{A}}$, there there exists $\p^*\in\R^d$ such that $\inf_{\p}P_\nu(f_{\p})=P_\nu(f_{\p^*})$, where $\mathcal{P}^o_{\mathbf{A}}$ denotes the interior of $\mathcal{P}_{\mathbf{A}}$.
\end{lem}

\begin{proof}
	Suppose that $\alpha\in\mathcal{P}^o_{\mathbf{A}}$. Then there exists an $\eta>0$ such that for every $\underline{p}\in\R^d$ with $|\underline{p}|=1$ there exists $\mu\in\mathcal{M}_\nu(\Gamma_{\mathbf{A}})$ such that $\int fd\mu-\alpha=\eta\underline{p}$. Thus, for every $c>0$
	$$
	P_\nu(f_{c\p})\geq h_\mu^\xi+\int\langle c\underline{p},f-\alpha\rangle d\mu\geq c\eta|\underline{p}|^2=c\eta.
	$$
	Thus, $\lim_{|\p|\to\infty}P_\nu(f_{\p})=\infty$ and by the convexity of the conditional pressure Lemma~\ref{lem:convex}, we get the statement.
\end{proof}

\begin{lem}\label{lem:int}
	Let $\p^*\in\R^d$ be such that $\inf_{\p}P_\nu(f_{\p})=P_\nu(f_{\p^*})$ and let $\mu_{\p^*}$ be the conditional equilibrium defined in Theorem~\ref{thm:equillibrium}. Then
	$$
	\int \phi d\mu_{\p^*}=\alpha.
	$$
\end{lem}

\begin{proof}
	Let us argue by contradiction. Suppose that $\int \phi d\mu_{\p^*}\neq\alpha$. Let $\underline{q}=\frac{\int\phi d\mu_{\p^*}-\alpha}{|\int\phi d\mu_{\p^*}-\alpha|}$.
	
	Observe that for any $\p_1,\p_2\in\R^d$ and $t\in\R$, $t f_{\p_1}+(1-t) f_{\p_2}=f_{t\p_1+(1-t)\p_2}$. Hence, by Theorem~\ref{thm:differentiable}, the function $p\colon t\mapsto P_\nu(f_{(1-t)\p^*+(\p^*+\underline{q})t})$ is differentiable at $t=0$, moreover,
	$$
	p'(0)=\int f_{\p^*+\underline{q}}-f_{\p^*}d\mu_{\p^*}.
	$$
	But $p$ has a minimum at $t=0$ so
	$$
	0=	p'(0)=\int f_{\p^*+\underline{q}}-f_{\p^*}d\mu_{\p^*}=\langle\underline{q},\int\phi d\mu_{\p^*}-\alpha\rangle=\left|\int\phi d\mu_{\p^*}-\alpha\right|,
	$$
	which is a contradiction.
\end{proof}



\begin{proof}[Proof of Theorem~\ref{thm:typmain}]
	It is enough to show that for every $\alpha\in\mathcal{P}^o_{\mathbf{A}}$ and $\nu$-almost every $\ww$
	$$
	h_{\rm top}(E_\ww(\alpha))\geq h_{\mu_{\p^*}}-h_\nu,
	$$
	where $\mu_{\p^*}$ is the conditional equilibrium of $P_\nu(f_{\p^*})=\inf_{\p\in\R^d}P_\nu(f_\p)$ defined in Theorem~\ref{thm:equillibrium}. Indeed, Theorem~\ref{thm:entconv}, Lemma~\ref{lem:int} and Theorem~\ref{thm:equillibrium} imply that
	$$
	h_{\mu_{\p^*}}-h_\nu=h_{\mu_{\p^*}}^\xi=h_{\mu_{\p^*}}^\xi+\int f_{\p^*}d\mu_{\p^*}=P_\nu(f_{\p^*})=\inf_{\p\in\R^d}P_\nu(f_\p).
	$$
	The upper bound follows by equation \eqref{eq:condpres2} and Lemma~\ref{lem:ub}.
	
	Let $\mu_\ww^\xi$ be the family of conditional measures with respect to the partition $\xi$ and $\mu_{\p^*}$ defined by Rohlin's Disintegration Theorem. By Theorem~\ref{thm:entconv},
	$$
	\lim_{n\to\infty}\frac{-1}{n}\log\mu_{\ww}^\xi([\ii|_n])=h_{\mu_{\p^*}}-h_\nu\text{ for $\mu_{\p^*}$-almost every $(\ww,\ii)\in\Gamma_{\mathbf{A}}$.}
	$$
	By Egorov's theorem, for every $\varepsilon>0$ there exists a set $J_1\subset\Gamma_{\mathbf{A}}$ and a constant $C>0$ such that $\mu_{\p^*}(J_1)>1-\varepsilon$ and for every $(\ww,\ii)\in J_1$ and $n\geq1$
	$$
	\mu_{\ww}^\xi([\ii|_n])\leq Ce^{-n(h_{\mu_{\p^*}}-h_\nu-\varepsilon)}.
	$$
	Since $1-\varepsilon<\mu_{\p^*}(J_1)=\int\mu_\ww^\xi(J_1)d\nu(\ww)$, by Markov's inequality, we get that
	$$
	\nu(\{\ww\in\Omega: \mu_\ww^\xi(J_1\cap\xi(\ww))>1-\sqrt{\varepsilon})>1-\sqrt{\varepsilon}.
	$$
	By Birkhoff's Ergodic Theorem and Lemma~\ref{lem:int},
	$$
	\lim_{n\to\infty}\frac{1}{n}\sum_{k=0}^{n-1}f(\sigma^k\ww,\sigma^k\ii)=\alpha.
	$$
	Hence, there exists $J\subset J_1$ such that $\nu(J_1\setminus J)=0$ and for every $\ww\in J$, $\mu_\ww^\xi(E_\ww(\alpha)\cap J_1)>1-\sqrt{\varepsilon}$. Thus, by Lemma~\ref{lem:Forstman} for every $\ww\in J$
	$$
	h_{\rm top}(E_\ww(\alpha))\geq h_{\rm top}(E_\ww(\alpha)\cap J_1)\geq h_{\mu_{\p^*}}-h_\nu-\varepsilon.
	$$
	Since $\varepsilon>0$ was arbitrary, the statement follows.

Finally, let $\widetilde{\mu}$ be the ergodic $\sigma$-invariant measure on $\Sigma_{\mathbf{A}}$ such that $h_{\widetilde{\mu}}=h_{\rm top}(\Sigma_{\mathbf{A}})$. Then for $\alpha_0=\iint f(\ww,\ii)d\widetilde{\mu}(\ii)d\nu(\ww)$ we get $h_{\rm top}(E_\ww(\alpha_0))\geq h_{\rm top}(\Sigma_{\mathbf{A}})$ for $\nu$-almost every $\ww$.
\end{proof}

\begin{proof}[Proof of Theorem~\ref{cor:main}]
	Let $I$ be the domain of the map
	$$
	p\colon\alpha\mapsto\inf_{p\in\R}P_\nu(p\cdot (f-\alpha))=\inf_{p\in\R}\left(P_\nu(p\cdot f)-p\alpha\right).
	$$
	If $I$ is empty or a single point then there is nothing to prove, so we might assume that $I$ has non-empty interior. By Theorem~\ref{thm:differentiable}, the map $p\mapsto P_\nu(p\cdot f)$ is differentiable and by Lemma~\ref{lem:convex}, the derivative $p\mapsto P_\nu'(p\cdot f)=\int fd\mu_p$ is increasing. Hence, $I=[\lim_{p\to-\infty}P_{\nu}'(pf),\lim_{p\to\infty}P_{\nu}'(pf)]$. Moreover, the map $\alpha\mapsto p(\alpha)$ is concave and continuous over $I$.
	
	By Theorem~\ref{thm:typmain}, for every $\alpha\in I^o$ and $\nu$-almost every $\ww$, $h_{\rm top}(E_\ww(\alpha))=p(\alpha)$. Then by Fubini's Theorem, for $\nu$-almost every $\ww$ and Lebesgue almost every $\alpha\in I^o$, $h_{\rm top}(E_\ww(\alpha))=p(\alpha)$.
	
	Using Theorem~\ref{thm:cont} with the choice $\phi_i(\ii):= f(\sigma^i\ww,\ii)$, the map $\alpha\mapsto h_{\rm top}(E_\ww(\alpha))$ is continuous for every $\ww\in\Omega$. This together with the continuity of the map $\alpha\mapsto p(\alpha)$ implies that $p(\alpha)\equiv h_{\rm top}(E_\ww(\alpha))$ over $I$ for $\nu$-almost every $\ww$.
\end{proof}

\section{Frequency regular sequences}\label{sec:Frequency regular sequences}

In the rest of the paper, we assume that $\Sigma_{\mathbf{A}}=\Sigma$, that is, we need to work on the full shift.
In this section, we establish the connection between $\nu$-typical and frequency regular sequences and prove
Theorem~\ref{thm:goal}.
The proof of our main theorem relies on the following construction, which first appeared in Rams \cite{R2}.

Let $\ww,\ww'\in\Omega$ be two $\underline{q}$-frequency regular sequences with the same frequency.
We define a permutation $\gamma$ on $\N$ such that

\begin{equation}\label{def:gamma}
	\begin{split}
	\gamma(k)&=\ell\text{ if $\omega_k$ is the $n$th appearance of the symbol of $\omega_k$ in $\ww$ }\\
	&\qquad\text{then $\ell$ is the position of the $n$th appearance of $\omega_k$ in $\ww'$.}
	\end{split}
\end{equation}

More precisely, let
$$  M_{n,\lambda_i}(\ww)=\min\{k\geq1: \#\{1\leq j\leq k:w_j=\lambda_i\}=n\}  $$
and
$$  P_{k}(\ww)=\#\{1\leq i\leq k:w_i=w_k\}.  $$
Then
$$  \gamma(k)=M_{P_k(\ww),w_k}(\ww').  $$

By the definition of $\gamma(k)$, we have $w_k=w_{\gamma(k)}'$. Finally, we set the map
\begin{equation}\label{def:almlip}
	G_{\ww,\ww'}(\ii):=(i_{\gamma(1)},i_{\gamma(2)},\ldots).
\end{equation}

\begin{lem}\label{lem:hold}
	For $\ww,\ww'\in\Omega$ as above, for every $\alpha<1$ there
	exists $C>0$
	such that for every $\ii,\jj\in\Sigma$
	\begin{equation}
		\label{eq:hold} d(G_{\ww,\ww'}(\ii),G_{\ww,\ww'}(\jj))\leq Cd(\ii,\jj)^{\alpha}.
	\end{equation}
	Moreover, $G_{\ww,\ww'}\circ G_{\ww',\ww}(\ii)=\ii$.
\end{lem}
\begin{proof}
	The construction clearly implies that $G_{\ww,\ww'}\circ G_{\ww',\ww}$ is the identity map on $\Sigma$.
	
	Since $\ww,\ww'\in\Omega$ are frequency regular sequences, we have that $\lambda_i$ appears infinitely often in $\ww,\ww'$.
	Thus, for every $n\geq1$ we can define $m_n$ such that $m_n$ is the smallest positive integer such that
	$\{1,\ldots,n\}\subseteq\{\gamma(1),\ldots,\gamma(m_n)\}$. Hence, for every $n\geq1$
	
	$$  \text{ if }d(\ii,\jj)= e^{-m_n-1}\text{ then }d(G_{\ww,\ww'}(\ii),G_{\ww,\ww'}(\jj))= e^{-n-1}.  $$
	
	Thus, to prove \eqref{eq:hold}, it is enough to show that
	\begin{equation}\label{eq:limneed}
		\lim_{n\to\infty}\frac{m_n}{n}=1.
	\end{equation}
	
	Clearly $m_n\geq n$, so $\liminf_{n\to\infty}\frac{m_n}{n}\geq1$. By the definition of $m_n$, for every $i=1,\ldots,N$,
	\begin{equation}\label{eq:ineq1}
		\#\{1\leq k\leq m_n:w_k=\lambda_i\}\geq\#\{1\leq k\leq n:w_k'=\lambda_i\},
	\end{equation}
	
	and there exists (at least one) $j=j(n)$ such that
	\begin{equation}\label{eq:ineq2}
		\#\{1\leq k\leq m_n:w_k=\lambda_j\}=\#\{1\leq k\leq n:w_k'=\lambda_j\}.
	\end{equation}
	
	By frequency regularity, for every $0<\varepsilon<\min_iq_i/2$ there exists $N\geq1$ such that for every $n\geq N$
	
	$$  \left|\frac{\#\{1\leq k\leq n:w_k=\lambda_i\}}{n}-q_i\right|,\left|\frac{\#\{1\leq k\leq n:w_k'=\lambda_i\}}{n}-q_i\right|<\varepsilon.  $$
	
	Hence, by \eqref{eq:ineq2} for every $n\geq N$
	
	\[  
	\begin{split}  
		\frac{m_n}{n}(q_{j(n)}-\varepsilon)&\leq\frac{m_n}{n}\frac{\#\{1\leq k\leq m_n:w_k=\lambda_{j(n)}\}}{m_n}\\  
		&=\frac{\#\{1\leq k\leq n:w_k'=\lambda_{j(n)}\}}{n}\leq q_{j(n)}+\varepsilon.  
	\end{split}  
	\]
	
	Thus, for every $n\geq N$, $\frac{m_n}{n}\leq 1+4\varepsilon$.
\end{proof}

\begin{prop}\label{lem:equal}
	For every  $\underline{q}$-frequency regular sequences $\ww,\ww'$ with the same frequency
\begin{equation}\label{eq:enteq}
h_{\rm top}(E_\ww(\alpha))=h_{\rm top}(E_{\ww'}(\alpha)).
\end{equation}
\end{prop}

\begin{proof}
Let  $\ww,\ww'$ be $\underline{q}$-frequency regular sequences. Let $G_{\ww,\ww'}$ be the map defined in \eqref{def:almlip}. It is enough to show that
\begin{equation}\label{eq:cont}
G_{\ww,\ww'}(E_{\ww'}(\alpha))\subseteq E_{\ww}(\alpha).
\end{equation}
Indeed, by \eqref{eq:hold},
$$
h_{\rm top}(E_{\ww'}(\alpha))=h_{\rm top}(G_{\ww',\ww}\circ G_{\ww,\ww'}(E_{\ww'}(\alpha))\leq h_{\rm top}(G_{\ww,\ww'}(E_{\ww'}(\alpha))\leq h_{\rm top}(E_{\ww}(\alpha)).
$$
The other inequality follows by symmetry.

Let $\gamma\colon\N\mapsto\N$ be the map defined in \eqref{def:gamma}. Let us define $p_n$ as the largest non-negative integer such that $\{1,\ldots,p_n\}\subseteq\{\gamma(1),\ldots,\gamma(n)\}$. In other words, $p_n=\min\{k\geq1:k\notin\{\gamma(1),\ldots,\gamma(n)\}\}-1$. Similarly to \eqref{eq:limneed} one can show that
\begin{equation}\label{eq:limneed2}
\lim_{n\to\infty}\frac{p_n}{n}=1.
\end{equation}
Let $\ii\in E_{\ww'}(\alpha)$. Then by \eqref{eq:limneed2}
\[
\begin{split}
\frac{1}{n}\sum_{k=0}^{n-1}f(\sigma^k\ww,\sigma^{k}G_{\ww,\ww'}(\ii))&=\frac{1}{n}\sum_{k=0}^{n-1}f_{w_k,i_{\gamma(k)}}\\
&=\frac{1}{n}\sum_{k=0}^{n-1}f_{w_{\gamma(k)}',i_{\gamma(k)}}\\
&=\frac{1}{n}\sum_{k=0}^{p_n-1}f_{w_k',i_{k}}+\frac{1}{n}\sum_{\substack{k=0 \\ \gamma(k)>p_n}}^{n-1}f_{w_{\gamma(k)}',i_{\gamma(k)}}\\
&\leq\frac{p_n}{n}\frac{1}{p_n}\sum_{k=0}^{p_n-1}f_{w_k',i_{k}}+\frac{n-p_n}{n}\max_{i,j}f_{i,j}\to\alpha,
\end{split}
\]
as $n\to\infty$.
Similarly,
$$
\frac{1}{n}\sum_{k=0}^{n-1}w_k\phi(\sigma^{k}G_{\ww,\ww'}(\ii))\geq\frac{p_n}{n}\frac{1}{p_n}\sum_{k=0}^{p_n-1}f_{w_k',i_{k}}+\frac{n-p_n}{n}\min_{i,j}f_{i,j}\to\alpha
$$
as $n\to\infty$. Hence, $G_{\ww,\ww'}(\ii)\in E_{\ww}(\alpha)$ which verifies \eqref{eq:cont}.
\end{proof}

\begin{proof}[Proof of Theorem~\ref{thm:goal}] Let $\nu$ be the Bernoulli measure associated to the weights $\underline{q}=(q_1,\ldots,q_N)$. Simple calculations show that the conditional pressure $P_\nu(\langle f-\alpha,\p\rangle)$ defined in \eqref{eq:condpresdefdef} equals to $P_{\underline{q}}(\langle f-\alpha,\p\rangle)$ in \eqref{eq:simplepres}.
	
	Hence, by applying Theorem~\ref{thm:typmain} we get that for every $\alpha$ and $\nu$-almost every $\ww$
		\[
	\begin{split}
	h_{\rm top}(E_{\ww}(\alpha))&=\sup\{h_\mu:\mu\in\mathcal{E}_\nu(\Gamma)\text{ and }\sum_{i,j}^{K,N}f_{j,i}\mu([j,i])=\alpha\}-h_\nu\\
	&=\inf_{\p\in\R^d}P_{\underline{q}}(\langle\p,f-\alpha\rangle).
	\end{split}
	\]
	By convexity, $\inf_{\p\in\R^d}P_{\underline{q}}(\langle\p,f-\alpha\rangle)$ is attained at $\p^*$. By Theorem~\ref{thm:equillibrium}, we know that the measure $\mu_{\p^*}$ where the supremum is attained can be chosen such that
	$$
C^{-1}\frac{Y(f_{\p^*},\ww,\ii)}{W_{|\ww|}(f_{\p^*},\ww)}\nu([\ww])\leq\mu_{\p^*}([\ww,\ii])\leq C\frac{Y(f_{\p^*},\ww,\ii)}{W_{|\ww|}(f_{\p^*},\ww)}\nu([\ww]),
	$$
	hold for some uniform constant $C>0$, where $f_{\p^*}(\ww,\ii)=\langle\p,\lambda_{w_0}\phi_{i_0}-\alpha\rangle$.	However, in this case,
	$$
	\eta([\ww,\ii])=\frac{Y(f_{\p^*},\ww,\ii)}{W_{|\ww|}(f_{\p^*},\ww)}\nu([\ww])=\prod_{k=0}^{|\ww|-1}\frac{q_{w_k}e^{\langle\p^*,\lambda_{w_k}\phi_{i_k}-\alpha\rangle}}{\sum_{i=1}^Ke^{\langle\p^*,\lambda_{w_k}\phi_{i}-\alpha\rangle}}
	$$
	is clearly an ergodic Bernoulli measure on $\Gamma$, since $\mu_{\p^*}$ is equivalent to $\eta$, we have $\eta=\mu_{\p^*}$. This shows that the supreme is attained at Bernoulli measures.
	
	Finally, since $\nu$-almost every sequence $\ww$ is $\underline{q}$-frequency regular, the statement follows by Proposition~\ref{lem:equal}.
\end{proof}

\begin{proof}[Proof of Theorem~\ref{thm:irreg2}]
Since the function $g(i)=\sum_{j=1}^Nq_jf_{j,i}$ is not constant by assumption, the possible values of $\alpha$, for which $\sum_{i,j}p_{j,i}f_{j,i}=\alpha$ and $\sum_ip_{j,i}=q_j$ form a nontrivial closed interval. Hence, the statement follows by Theorem~\ref{thm:contgen}.

\end{proof}

Now we finish the paper by showing the necessity of the frequency regular condition to have non-degenerate spectrum. Example~\ref{ex:alter} follows by the next example.

\begin{ex}\label{ex:degen}
	There exists a sequence $\ww\in\{0,1\}^\N$, which is not frequency regular, such that the following holds: For every continuous potential $\varphi\colon\{0,1\}^\N\mapsto\R$, $E_\ww(\alpha)=\emptyset$ for every $\alpha\in\R\setminus\{0\}$.
	
	
	Moreover, if $\varphi$ depends only the first symbol then $E_\ww(0)\neq\emptyset$ if and only if $\varphi_0\varphi_1\leq0$, moreover if additionally $\varphi_0\neq-\varphi_1$ then $h_{\rm top}(E_\ww(0))<\log 2$.
\end{ex}

\begin{proof}[Proof of Example~\ref{ex:degen}]
	First, let us define the sequence $\ww\in\{0,1\}^\N$. Let $\{M_n\}_{n=0}^\infty$ be a fast increasing sequence, that is, suppose that $2M_{n}<M_{n+1}$ for every $n\geq0$ and $\lim_{n\to\infty}\frac{\sum_{j=1}^nM_j}{M_{n+1}}=0$. Let $\ww:=(w_0,w_1,\ldots)$, where
	$$
	w_k=\begin{cases}
	0 & \text{if }2M_{n-1}< k\leq M_n,\\
	1 & \text{if }M_n<k\leq 2M_n.
	\end{cases}
	$$
	
	Clearly, $\ww$ is not frequency regular. Moreover, since for every $\ii\in\Sigma$
	$$
	\left|\frac{1}{M_n}\sum_{k=0}^{M_n}w_k\varphi(\sigma^k\ii)\right|\leq\frac{\max_{\ii\in\Sigma}|\varphi(\ii)|\sum_{\ell=0}^{n-1}M_\ell}{M_n}\to0\text{ as }n\to\infty,
	$$
	we get that $E_\ww(\alpha)=\emptyset$ for every $\alpha\neq0$. On the other hand, if $\min_{\ii\in\Sigma}\varphi(\ii)>0$ then
	$$
	\frac{1}{2M_n}\sum_{k=0}^{2M_n}w_k\varphi(\sigma^k\ii)\geq\frac{\min_{\ii\in\Sigma}\varphi(\ii)\sum_{\ell=0}^{n}M_\ell}{2M_n}\to\frac{\min_{\ii\in\Sigma}\varphi(\ii)}{2}\text{ as }n\to\infty,
	$$
	so $E_\ww(0)=\emptyset$ as well. Similarly, $E_\ww(0)=\emptyset$ also in the case if $\max_{\ii\in\Sigma}\varphi(\ii)<0$.
	
	Now, suppose that $\varphi(\ii)=\varphi_{i_0}$. Using the previous calculations if $\varphi_0\varphi_1>0$ then $E_\ww(\alpha)=\emptyset$ for every $\alpha\in\R$. So we may assume that $\varphi_0\varphi_1\leq0$. If $\varphi_0=0$ then the sequence $(0,0,\ldots)$ belongs to $E_\ww(0)$, so let us assume $\varphi_0<0<\varphi_1$. Then let us define the sequence $\ii$ inductively by the rule $i_{m+1}=0$ if and only if $A_m(\ii)>0$. Thus, $\ii$ belongs to $E_\ww(0)$.

	Additionally, suppose that $\varphi_0\neq-\varphi_{1}$. For every $m$, let $n_m$ be such that $M_{n_m}< m\leq M_{n_m+1}$. By the definition of $\ww$ we get that
	$$A_m(\ii)=\frac{1}{m}\sum_{k=0}^mw_k\varphi_{i_k}=\frac{M_{n_m}\sum\limits_{k=0}^{M_{n_m}}w_k\varphi_{i_k}}{mM_{n_m}}+\frac{\sum\limits_{k=M_{n_m}+1}^{\min\{m,2M_{n_m}\}}\varphi_{i_k}}{m}.
	$$ Since $\frac{\sum\limits_{k=0}^{M_{n_m}}w_k\varphi_{i_k}}{M_{n_m}}\to0$ as $m\to\infty$ and $\frac{M_{n_m}}{m}$ is bounded, we get $A_m(\ii)\to0$ if and only if
	$$
	\frac{\sum_{\ell=0,1} \#\{M_{n_m}<k\leq \min\{m,2M_{n_m}\}:i_k=\ell\}\varphi_\ell}{m}\to0.
	$$
	In particular, $A_m(\ii)\to0$ implies that
	\begin{equation}\label{eq:thiis}
	\frac{\#\{M_{n}<k\leq 2M_{n}:i_k=0\}}{M_n}\to\frac{\varphi_1}{\varphi_1-\varphi_0}\text{ as }n\to\infty.
	\end{equation}
	Denote $F$ the set of all $\ii\in\Sigma$, which satisfy \eqref{eq:thiis}. Then $h_{\rm top}(E_\ww(0))\leq h_{\rm top}(F)$.
	
	For short, let $p=\frac{\varphi_1}{\varphi_1-\varphi_0}$. Well known (for example, it is an application of Stirling's formula) that there exists $K(p)>|H'(p)|$, where $H(p) = -p\log p -(1-p)\log(1-p)$ such that for every $\varepsilon>0$ there exists $L\geq1$ such that for every $n\geq L$
	$$
	\#\left\{\ii\in\{0,1\}^n:\left|\frac{\#\{0<k\leq n:i_k=0\}}{n}-p\right|<\varepsilon\right\}\leq e^{(-p\log p-(1-p)\log(1-p)+K(p)\varepsilon)n}
	$$ and by \eqref{eq:topentbasic},
	\[
	\begin{split}
	h_{\rm top}(F)&\leq\liminf_{n\to\infty}\frac{1}{2M_n}\log\#\left\{\ii\in\Sigma_{2M_n}:F\cap[\ii]\neq\emptyset\right\}\\
	&\leq\lim_{n\to\infty}\frac{1}{2M_n}\log\prod_{k=1}^n2^{M_{k}-2M_{k-1}}e^{(-p\log p-(1-p)\log(1-p)+K(p)\varepsilon)M_{k}}\\
	&=\frac{\log2-p\log p-(1-p)\log(1-p)+K(p)\varepsilon}{2}.
	\end{split}
	\]
	Since $\varepsilon>0$ was arbitrary and by assumption $p\neq1/2$, we get $h_{\rm top}(F)<\log2$, which completes the proof.
\end{proof}

\end{document}